\documentclass[12pt,reqno,a4paper]{amsart}

\usepackage[T1]{fontenc}

\usepackage{libertinus}
\usepackage{newtxmath}

\usepackage{microtype}
\DisableLigatures[f]{encoding = *, family = * }

\usepackage{amsxtra,mathrsfs,mathtools,amsthm,thmtools}
\usepackage[colorlinks=true, pdfstartview=FitW,hypertexnames=false]{hyperref}

\usepackage[capitalise]{cleveref}

\AtBeginDocument{\frenchspacing}

\usepackage{enumitem}
\setlist[enumerate,1]{font=\upshape}

\setlist{nosep}

\allowdisplaybreaks[4]

\newtheorem{theorem}{Theorem}[section]
\newtheorem{proposition}[theorem]{Proposition}
\newtheorem{lemma}[theorem]{Lemma}
\newtheorem{corollary}[theorem]{Corollary}

\newtheorem{remark}[theorem]{Remark}

\newtheorem{claim}[theorem]{Claim}

\theoremstyle{definition}
\newtheorem{definition}[theorem]{Definition}

\numberwithin{equation}{section}

\title[Variational relations of topological pressure]{Variational relations of topological pressure for nonautonomous dynamical systems}

\author{Chang-Bing Li}
\address{Chang-Bing Li: School of Mathematics and Statistics, Hainan Normal University, Haikou, Hainan 571158, P. R. China \&
Department of Mathematics, Shantou University, Shantou, Guangdong 515063, P. R. China}
\email{christiesyp@gmail.com}



\date{\today}
\begin{document}
\begin{abstract}
This manuscript investigates the relationship between various notions of topological pressures and their corresponding measure-theoretic pressures for nonautonomous dynamical systems, using the framework of the Carath{\'e}odory-Pesin structure. 
We establish a pressure distribution principle for the Pesin topological pressure and prove a Billingsley type theorem for the packing topological pressure. 
Based on these results, we derive a variational principle for packing topological pressure of nonautonomous dynamical systems, revealing a variational connection between packing pressure and the measure-theoretic upper local pressure. 
Additionally, we explore the applicability of our findings to a typical nonautonomous dynamical system in which the sequence of continuous selfmaps preserves the same Borel probability measures. 
Finally, we obtain an upper bound for the packing topological pressure on the set of generic points.
\end{abstract}

\maketitle

\section{Introduction}
The topological entropy and the measure-theoretic entropy are two fundamental invariants in the study of dynamical systems.
Topological entropy quantifies the overall exponential complexity of the orbit structure, whereas the measure-theoretic entropy characterizes the information growth rate of evolution.
These entropies play a central role in dynamical systems theory.
A notable link between the topological entropy and measure-theoretic entropy is provided by the celebrated variational principle, which asserts that the topological entropy equals the supremum of the measure-theoretic entropy taken over all invariant Borel probability measures.
The topological pressure, as a natural extension of the topological entropy, is another significant concept in dynamical systems, with deep connections to thermodynamic formalism and dimension theory.
An analogous variational principle has been established, relating topological pressures and measure-theoretic entropies \cite{Wal1975,Wal2000a}.

In 1973, Bowen \cite{Bow1973} introduced a notion of the topological entropy for noncompact subsets by employing a method analogous to Hausdorff dimension.
Since then, the earlier notion of topological entropy, defined via spanning sets or separated sets, has been referred to as the classical topological entropy, whereas Bowen's novel entropy is commonly known in the literature as the Bowen topological entropy.
Pesin and Pitskel' \cite{PP1984} extended Bowen's approach to topological pressure and established a variational principle for topological pressures.
Moreover, they showed that the Pesin topological pressure coincides with the Bowen topological entropy when the potential function is identically zero.
Later, Pesin \cite{Pes1997} systematically summarized these theories and developed the well-known Carath{\'e}odory-Pesin structure.
This theory has inspired a great deal of subsequent research; accordingly, we provide a brief overview of the most relevant results, to the best of our knowledge.

Takens and Verbitskiy \cite{TV2003} established a variational principle for the Pesin topological entropy on level sets.
Feng and Huang \cite{FH2012} introduced the concept of weighted topological entropy and proved a variational principle for both Bowen topological entropy and packing topological entropy on subsets, involving the measure-theoretic lower and upper entropies.
Subsequently, Tang et al. \cite{TCZ2015} employed Feng-Huang's method to study the weighted topological pressure and derived a corresponding variational principle for Pesin topological pressure.
Feng and Huang \cite{FH2016} further investigated a variational principle for weighted topological pressure, which subsequently inspired the theory of mean Hausdorff dimension and mean dimension \cite{Tsu2023}.
Using the framework of Carath{\'e}odory-Pesin structures, Shen et al.~\cite{SXZ2020} introduced the $\mathbf{a}$-weighted topological entropy of a flow on noncompact sets and the $\mathbf{a}$-weighted measure-theoretic entropy, and further established a related variational principle.
Wang \cite{Wan2021b} examined two new variational principles for Bowen topological entropy and packing topological entropy by defining four measure-theoretic quantities associated with Borel probability measures, which characterize the upper and lower Katok measure-theoretic entropies.
Recently, Zhong and Chen \cite{ZC2023} studied multiple notions of topological pressures and measure-theoretic pressures and obtained two variational principles for Pesin topological pressure and packing topological pressure.
These variational principles relate to the measure-theoretic lower and upper local pressure, as well as the Pesin and packing pressure of Borel measures.

The topological entropy for nonautonomous dynamical systems (NDS) was originally studied by Kolyada and Snoha \cite{KS1996}.
Building on this, Zhu et al. \cite{ZLXZ2012} investigated both the topological entropy and a specific type of measure-theoretic entropy for NDS, deriving an inequality relating these two notions.
Kawan \cite{Kaw2014} further extended such results by showing that, in more general systems, the measure-theoretic entropy can be bounded above by its topological entropy.
Meanwhile, Liu et al.~\cite{LQX2020} studied the topological entropy of an NDS and its connection to the entropy of its induced system.
Subsequently, Huang et al. \cite{HWZ2008} and Kong et al. \cite{KCL2015} delved into the topological pressure of NDS and its fundamental properties in greater depth.
Yang and Guo \cite{YG2015} adopted the approach outlined in \cite{ZLXZ2012} to establish a variational principle for the classical topological pressure.
However, due to the intrinsic characteristics of NDS, this result falls short of the expected variational principle for classical topological pressure.
Li \cite{Li2015} presented a systematic overview of the relationships among classical topological entropy, Bowen topological entropy and Pesin topological entropy for NDS.
In a comprehensive study, Ju and Yang \cite{JY2021} examined classical properties of the Pesin topological entropy for NDS.
Yang \cite{Yan2021} introduced the notions of upper and lower Katok measure-theoretic entropies as well as measure-theoretic entropy, and established a variational principle for Bowen topological entropy.
Later, Li and Ye \cite{LY2023} further analyzed the interplay among distance entropy, Bowen topological entropy, Pesin topological entropy, classical topological entropy and Hausdorff dimension.
Moreover, they established a sufficient condition under which these notions of topological entropy are equivalent.
By employing the approach outlined in Feng and Huang \cite{FH2012}, Zhang and Zhu~\cite{ZZ2023} and Sarkooh \cite{Naz2024} demonstrated a variational principle for packing topological entropy and Bowen topological pressure, respectively.
Recently, Wang et al.~\cite{WYZ2025a} disproved the variational principle for classical topological entropy in the setting in which the sequence of continuous selfmaps preserves the same measure.

The aforementioned studies have led to several natural questions regarding the relationships among various notions of topological pressure, and the variational connections between packing topological pressure and measure-theoretic pressure on subsets of NDS.
To tackle these issues, various types of topological pressures have been investigated, including Pesin topological pressure, packing topological pressure, and the lower and upper capacity topological pressure, along with their corresponding measure-theoretic pressures.
In contrast to classical autonomous systems, the topological properties of NDS depend on a sequence of maps rather than a single map, and the finite sum of potential functions generally fails to satisfy the additive condition.
As a consequence, many well-established results from autonomous dynamics, such as the Birkhoff ergodic theorem and the Shannon-McMillan-Breiman theorem, do not readily apply to NDS.
On a related note, nonautonomous iterated function systems (NIFS for short) generalize both finitely generated semigroups and NDS by allowing the set of transformations to vary with time \cite{GS2019,CL2023c,SLZ2024,JLY2024}.
The Pesin topological pressure and the associated variational principle for NIFS have also been studied extensively \cite{CL2023c,JLY2024}.
Nevertheless, this manuscript focuses on extending the notion of packing topological pressure and its variational principle to NDS.
As demonstrated in Theorem \ref{firsttheorem} below, the packing topological pressure is not, in general, equivalent to the Pesin topological pressure for NDS.
In accordance with the established framework of geometric measure theory, the forms of the Billingsley type theorem and the variational principle for packing topological pressure for NIFS are expected to coincide with those presented in this manuscript.
Moreover, it is anticipated that the results of this study can be extended to NIFS with relative ease, without encountering significant technical obstacles.

Throughout this paper, let $X$ be a compact metric space, $f_{1,\infty}$ a sequence of continuous selfmaps of $X$, $K$ a nonempty subset of $X$, and $\varphi$ a continuous function on $X$.
Then $(X,f_{1,\infty})$ forms a nonautonomous dynamical system (NDS for short), see the next section for detailed definitions.
We denote by $P(f_{1,\infty},K,\varphi)$, $\underline{CP}(f_{1,\infty},K,\varphi)$, $\overline{CP}(f_{1,\infty},K,\varphi)$, $P^B(f_{1,\infty},K,\varphi)$ and $P^P(f_{1,\infty},K,\varphi)$ the classical topological pressure, lower capacity pressure, upper capacity pressure, Pesin topological pressure and packing topological pressure of $\varphi$ on $K$ with respect to $f_{1,\infty}$, respectively.
Then the following \Cref{firsttheorem} (see also \Cref{relationship}) reveals the relationships among these notions of topological pressure.

\begin{theorem}\label{firsttheorem}
Let $(X,f_{1,\infty})$ be an NDS, $K$ a nonempty subset of $X$ and $\varphi\in C(X,\mathbb{R})$.
Then
\begin{enumerate}
\item[(1)] $P^B(f_{1,\infty},K,\varphi)\leq \underline{CP}(f_{1,\infty},K,\varphi)\leq \overline{CP}(f_{1,\infty},K,\varphi)$;

\item[(2)] $P^B(f_{1,\infty},K,\varphi)\leq P^P(f_{1,\infty},K,\varphi)\leq P(f_{1,\infty},K,\varphi)$;

\item[(3)] $P^P(f_{1,\infty},K,\varphi)\leq \overline{CP}(f_{1,\infty},K,\varphi)$;

\item[(4)] $\overline{CP}(f_{1,\infty},K,\varphi)=P(f_{1,\infty},K,\varphi)$.
\end{enumerate}
\end{theorem}

The immediate distribution principle for Pesin topological pressure (see also \Cref{distribution}) provides a useful method for deriving lower bounds of the Pesin topological pressure.
As illustrated in (2) of \Cref{firsttheorem}, these conditions are also applicable to the packing topological pressure.

\begin{theorem}
Let $(X,f_{1,\infty})$ be an NDS, $K$ a nonempty subset of $X$, $\varphi\in C(X,\mathbb{R})$, and $\mathcal{M}(X)$ the set of all Borel probability measures on $X$.
Suppose that there exist $\varepsilon>0$ and $s\geq 0$ such that one can find a sequence of Borel probability measures $\{\mu_k\} \subset \mathcal{M}(X)$, a constant $C>0$ and an integer $n\in \mathbb{N}$ satisfying the following conditions:
\begin{enumerate}
\item[(1)] $\mu_k(K)>0$ for all $k$;
\item[(2)] $\limsup_{k \to\infty} \mu_k(B_n(x,\varepsilon))\leq C\exp(-ns+S_n\varphi(x))$ for each Bowen ball $B_n(x,\varepsilon)$ such that $B_n(x,\varepsilon)\cap K\neq\emptyset$;
\item[(3)] the sequence $\{\mu_k\}$ admits at least one limit measure $\mu$ for which $\mu(K)>0$.
\end{enumerate}
Then $P^B(f_{1,\infty},K,\varphi)\geq s$.
\end{theorem}

For any $\mu\in \mathcal{M}(X)$, we denote by $\overline{P}_{\mu}(f_{1,\infty},x,\varphi)$ the upper measure-theoretic pressure of $x\in X$ with respect to $\mu$ and $\varphi$, and by $P^P(f_{1,\infty},K,\varphi)$ the packing pressure of $\mu$ with respect to $\varphi$, respectively.
Subsequently, the Vitali covering lemma (see also \Cref{vitali}) allows us to derive a Billingsley type theorem for packing topological pressure (see also \Cref{Billingsley}), which plays a crucial role in establishing the variational principle for packing topological pressure.

\begin{theorem}
Let $(X,f_{1,\infty})$ be an NDS, $K$ a nonempty subset of $X$ and $\varphi\in C(X,\mathbb{R})$. Then
\begin{enumerate}
 \item[(1)] if $\overline{P}_{\mu}(f_{1,\infty},x,\varphi)\leq s$ for all $x\in K$, then $P^P(f_{1,\infty},K,\varphi)\leq s$;

 \item[(2)] if $\overline{P}_{\mu}(f_{1,\infty},x,\varphi)\geq s$ for all $x\in K$ and $\mu(K)>0$, then $P^P(f_{1,\infty},K,\varphi)\geq s$.
 \end{enumerate}
\end{theorem}

We denote by $\overline{P}_{\mu}(f_{1,\infty},\varphi)$ the upper measure-theoretic pressure of $\mu$ with respect to $\varphi$ on $X$, which is the integral of $\overline{P}_{\mu}(f_{1,\infty},x,\varphi)$ on $X$.
The following \Cref{main1} (see also \Cref{main}) provides a variational principle for the packing topological pressure of NDSs, which is one of the main results in this manuscript.
Moreover, this theorem extends the results of \cite{ZZ2023} for the packing topological entropy of NDSs and \cite{ZC2023} for the packing entropy of autonomous dynamical systems.

\begin{theorem}\label{main1}
Let $(X,f_{1,\infty})$ be an NDS and $\varphi$ a continuous function on $X$.
If $K$ is a nonempty compact subset of $X$ and $P^P(f_{1,\infty},K,\varphi)\geq \|\varphi\|_{\infty}:=\sup_{x\in X}|\varphi(x)|$, then
\begin{align*}
P^P(f_{1,\infty},K,\varphi)
&=\sup\{\overline{P}_{\mu}(f_{1,\infty},K,\varphi):\mu\in \mathcal{M}(X),\mu(K)=1 \}\\
&=\sup\{P_{\mu}^P(f_{1,\infty},\varphi):\mu\in \mathcal{M}(X),\mu(K)=1 \}.
\end{align*}
\end{theorem}

This paper is organized as follows. The second section introduces preliminaries on topological pressures on subsets in nonautonomous dynamical systems (NDS). 
It also discusses measure-theoretic pressures and the associated Pesin, packing pressure, as well as the lower and upper capacity pressures associated with a given Borel probability measure. 
The third section provides a comprehensive analysis of the relationship among classical topological pressure, Pesin topological pressure, packing topological pressure, lower and upper topological pressures, and their corresponding measure-theoretic pressures. 
The fourth section constitutes the primary body of this manuscript, focusing on a Billingsley type theorem for Pesin (packing) topological pressure and a variational principle for packing topological pressure. 
The final section explores the applicability of these results to nonautonomous measure-preserving systems wherein the sequence of transformations preserves the same Borel probability measure. 
Additionally, an upper bound for the packing topological pressure on the set of generic points is derived.

\section{Various pressures on subsets}
Throughout this paper, let $X$ be a compact topological space, equipped with a sequence of continuous selfmaps $f_{1,\infty}=\{f_i\}_{i=1}^{\infty}$ (we do not assume that the maps $f_i$ are invertible).
For any $i,j\in\mathbb{N}$, let $f_i^0=\textnormal{id}_X$ be the identity map of $X$, and
\[
f_i^j:=f_{i+(j-1)}\circ\dots\circ f_{i+1}\circ f_i,\quad f_i^{-j}:=\left(f_i^j\right)^{-1}=f_i^{-1}\circ f_{i+1}^{-1}\circ\dots\circ f_{i+(j-1)}^{-1}.
\]
Then we call such system a \emph{nonautonomous dynamical system} (NDS for short), and denote it by $(X,f_{1,\infty})$. For any $x\in X$, the \emph{trajectory} of $x\in X$ is the sequence $\{x_n\}_{n=0}^{\infty}$, where we set $x_0=x$, and $x_{n+1}=f_{n+1}(x_n)$, that is, $x_n=f_1^n(x) $ for $n\geq 0$. The \emph{orbit of $x$} is the set
\[
Orb(x,f_{1,\infty})=\{x,f_1(x),f_1^2(x),\dots\}.
\]
We also denote by $f_{1,\infty}^n$ the sequence $\{f_{in+1}^n\}_{i=0}^{\infty}$ and by $f_{1,\infty}^{-1}$ the sequence $\{f_i^{-1}\}_{i=1}^{\infty}$ \cite{KS1996}.

Let $(X,d)$ be a compact metric space, $K$ a nonempty subset of $X$, and $f_{1,\infty}$ a sequence of continuous selfmaps of $X$. We define a new metric $d_n$ on $X$ by
\[
d_n(x,y)=\max_{0\leq i\leq n-1}d\left(f_1^i(x),f_1^i(y)\right).
\]
For any $x\in X$ and $\varepsilon>0$, the open and closed Bowen balls are given by
\[
B_n(x,\varepsilon)=\{y\in X:d_n(x,y)<\varepsilon \}\quad \text{and}\quad
\overline{B}_n(x,\varepsilon)=\{y\in X:d_n(x,y)\leq \varepsilon \},
\]
respectively. For any $\varepsilon>0$, a subset $E\subseteq K$ is called $(n,\varepsilon)$-separated if for any two distinct points $x,y\in E$, $d_n(x,y)>\varepsilon$.
A subset $F$ of $X$ is an $(n,\varepsilon)$-spanning set for $K$, or that it $(n,\varepsilon)$-spans $K\subseteq X$ if for each $x\in K$ there exists $y\in F$ such that $d_n(x,y)\leq \varepsilon$.
Let $s_n(f_{1,\infty},K,\varepsilon)$ be the maximal cardinality of an $(n,\varepsilon)$-separated set in $K$, and $r_n(f_{1,\infty},K,\varepsilon)$ the minimal cardinality of a set in $K$ which $(n,\varepsilon)$-spans $K$ \cite{KS1996}.

\subsection{Classical topological pressure}
Given an NDS $(X,f_{1,\infty})$ on the compact metric space $(X,d)$.
For any potential function $\varphi\in C(X,\mathbb{R})$ and $x\in X$, denote by $S_n\varphi(x)=\sum_{i=0}^{n-1}\varphi(f_1^i(x))$. For any $\varepsilon>0$, put
\begin{align*}
Q_n(f_{1,\infty},K,\varphi,\varepsilon)&=\inf\left\{\sum_{x\in F}\exp(S_n\varphi(x)): F \text{ is an }(n,\varepsilon)\text{-spanning set for }K \right\},\\
P_n(f_{1,\infty},K,\varphi,\varepsilon)&=\sup\left\{\sum_{x\in E}\exp(S_n\varphi(x)): E \text{ is an }(n,\varepsilon)\text{-separated set of }K \right\}.
\end{align*}
Then the following two limits exist:
\begin{align}
Q(f_{1,\infty},K,\varphi,\varepsilon)&=\limsup_{n\to\infty}\frac{1}{n}\log Q_n(f_{1,\infty},K,\varphi,\varepsilon), \label{Qn}\\
P(f_{1,\infty},K,\varphi,\varepsilon)&=\limsup_{n\to\infty}\frac{1}{n}\log P_n(f_{1,\infty},K,\varphi,\varepsilon).\nonumber
\end{align}

\begin{definition} \label{classicdefinition}
The \emph{classical topological pressure of $\varphi$ on the set $K$ with respect to the sequence $f_{1,\infty}$} is given by
\begin{equation}\label{classicalpressure}
P(f_{1,\infty},K,\varphi)=\lim_{\varepsilon\to 0}Q(f_{1,\infty},K,\varphi,\varepsilon)=\lim_{\varepsilon\to 0}P(f_{1,\infty},K,\varphi,\varepsilon).
\end{equation}
In particular, when $\varphi=0$, the classical topological pressure of $f_{1,\infty}$ on the set $K$ reduces to the classical topological entropy $h(f_{1,\infty},K)$, i.e.,
\[
h(f_{1,\infty},K)=P(f_{1,\infty},K,0).
\]
\end{definition}

The classical topological pressure can also be defined by open covers.
For open covers $\mathcal{U}_1, \mathcal{U}_2,\dots,\mathcal{U}_n$ of $X$, we denote
\[
\bigvee_{i=1}^n \mathcal{U}_i=\mathcal{U}_1 \vee\mathcal{U}_2 \vee\dots \vee\mathcal{U}_n=\{U_1\cap U_2\cap\dots\cap U_n: U_i\in \mathcal{U}_i, i=1,2,\dots,n \}.
\]
It is easy to see that $\bigvee_{i=1}^n \mathcal{U}_i$ is also an open cover of $X$.
For any open cover $\mathcal{U}$ of $X$, we denote by $|\mathcal{U}|=\max\{diam(U): U\in \mathcal{U}\}$, by $f_i^{-j}(\mathcal{U})=\{f_i^{-j}(U):U\in \mathcal{U}\}$ and by $\mathcal{U}_i^n=\bigvee_{j=0}^{n-1}f_i^{-j}(\mathcal{U})$.
For any nonempty subset $K\subseteq X$, let $\mathcal{U}|_K$ be the cover $\{U\cap K: U\in \mathcal{U}\}$ of $K$. We put
\[
q_n(f_{1,\infty},K,\varphi,\mathcal{U})=\inf\left\{\sum_{B\in \mathcal{B}}\inf_{x\in B} \exp(S_n\varphi(x))\right\},
\]
where the infimum is taken over all finite covers $\mathcal{B}$ of $\mathcal{U}_1^n|_K$.

\begin{lemma}\cite{LY2022}
Let $(X,f_{1,\infty})$ be an NDS, and $K$ a nonempty subset of $X$.
Then for any $\varphi\in C(X, \mathbb{R})$,
\[
P(f_{1,\infty},K,\varphi)=\sup\left\{\limsup_{n\to\infty}\frac{1}{n}\log q_n(f_{1,\infty},K,\varphi,\mathcal{U})\right\},
\]
where the supremum is taken over all open covers $\mathcal{U}$ of $X$.
\end{lemma}

It is known that
\[
  P(f_{1,\infty},K,\varphi)=\lim_{|\mathcal{U}|\to 0} \left\{\limsup_{n\to\infty}\frac{1}{n}\log q_n(f_{1,\infty},K,\varphi,\mathcal{U})\right\}.
\]

\subsection{Pesin topological pressure}
Let $(X,f_{1,\infty})$ be an NDS, $K$ a nonempty subset of $X$, and $\varphi$ a continuous function on $X$.
The Pesin (or Pesin-Pitskel') topological pressure on subsets was first introduced by Pesin and Pitskel' \cite{PP1984} and subsequently developed into the Carath{\'e}odory-Pesin theory by Pesin \cite{Pes1997}.
In \cite{PP1984}, the authors demonstrated that the Pesin topological pressure for autonomous dynamical system equals the Bowen topological entropy when the potential function becomes zero.
Consequently, we also denote the Pesin topological pressure by $P^B(f_{1,\infty},K,\varphi)$ throughout this manuscript.

Given $N\in \mathbb{N}$, $s\in \mathbb{R}$, $\varepsilon>0$, and $\varphi\in C(X,\mathbb{R})$, define
\begin{equation}\label{Mff}
    M(f_{1,\infty},K,\varphi,\varepsilon,s,N)=\inf\left\{ \sum_i \exp\big (-sn_i+S_{n_i}\varphi(x_i)\big)\right\},
\end{equation}
where the infimum is taken over all finite or countable collections of $\Gamma = \{B_{n_i}(x_i,\varepsilon)\}_i$ such that $x_i\in X$, $n_i\geq N$ and $\Gamma$ covers $K$, i.e., $K\subseteq \bigcup_i B_{n_i}(x_i,\varepsilon)$.

We also define
\[
R(f_{1,\infty},K,\varphi,\varepsilon,s,N)=\inf\left\{\sum_i\exp\big (-sN+S_N\varphi(x_i)\big )\right \},
\]
where the infimum is taken over all finite or countable collections of $\{B_N(x_i,\varepsilon)\}_i$ such that $x_i\in X$ and $K\subseteq \bigcup_i B_N(x_i,\varepsilon)$.

The quantity $M(f_{1,\infty},K,\varphi,\varepsilon,s,N)$ does not decrease as $N\to \infty$, hence the following limit exists:
\[
M(f_{1,\infty},K,\varphi,\varepsilon,s)=\lim_{N\to\infty} M(f_{1,\infty},K,\varphi,\varepsilon,s,N).
\]
We further let
\begin{align*}
\overline{r}(f_{1,\infty},K,\varphi,\varepsilon,s)&=\limsup_{N\to\infty} R(f_{1,\infty},K,\varphi,\varepsilon,s,N),\\
\underline{r}(f_{1,\infty},K,\varphi,\varepsilon,s)&=\liminf_{N\to\infty} R(f_{1,\infty},K,\varphi,\varepsilon,s,N).
\end{align*}
By the Carath{\'e}odory-Pesin theory \cite{Pes1997}, there exist critical values of the parameter $s$ such that $M(f_{1,\infty},K,\varphi,\varepsilon,s)$, $\overline{r}(f_{1,\infty},K,\varphi,\varepsilon,s)$ and $\underline{r}(f_{1,\infty},K,\varphi,\varepsilon,s)$ jump from $+\infty$ to $0$, we denote these values by $P^B(f_{1,\infty},K,\varphi,\varepsilon)$, $\overline{CP}(f_{1,\infty},K,\varphi,\varepsilon)$ and $\underline{CP}(f_{1,\infty},K,\varphi,\varepsilon)$, respectively.
\begin{align*}
 P^B(f_{1,\infty},K,\varphi,\varepsilon)&=\sup\{s: M(f_{1,\infty},K,\varphi,\varepsilon,s)=+\infty\}=\inf\{s: M(f_{1,\infty},K,\varphi,\varepsilon,s)=0\},\\
 \overline{CP}(f_{1,\infty},K,\varphi,\varepsilon)&=\sup\{s: \overline{r}(f_{1,\infty},K,\varphi,\varepsilon,s)=+\infty\}=\inf\{s: \overline{r}(f_{1,\infty},K,\varphi,\varepsilon,s)=0\},\\
\underline{CP}(f_{1,\infty},K,\varphi,\varepsilon)&=\sup\{s: \underline{r}(f_{1,\infty},K,\varphi,\varepsilon,s)=+\infty\}=\inf\{s: \underline{r}(f_{1,\infty},K,\varphi,\varepsilon,s)=0\}.
\end{align*}

\begin{definition}
Let $(X,f_{1,\infty})$ be an NDS, $K$ a nonempty subset of $X$ and $\varphi \in C(X,\mathbb{R})$. The \emph{Pesin topological pressure}, the \emph{upper and lower capacity topological pressure of the function $\varphi$ on the set $K$ with respect to $f_{1,\infty}$} are given by
\begin{align*}
P^B(f_{1,\infty},K,\varphi)&=\lim_{\varepsilon\to 0} P^B(f_{1,\infty},K,\varphi,\varepsilon),\\
\overline{CP}(f_{1,\infty},K,\varphi)&=\lim_{\varepsilon\to 0}\overline{CP}(f_{1,\infty},K,\varphi,\varepsilon),\\
\underline{CP}(f_{1,\infty},K,\varphi)&=\lim_{\varepsilon\to 0}\underline{CP}(f_{1,\infty},K,\varphi,\varepsilon),
\end{align*}
respectively. In particular, when $\varphi=0$, the Pesin topological pressure $P^B(f_{1,\infty},K,\varphi)$ reduces to the Bowen topological entropy, we denote it by $h_{top}^B(f_{1,\infty},K)$. 
Similarly, the upper and lower capacity topological entropies on $K$ are denoted by $\overline{Ch}(f_{1,\infty},K)$ and $\underline{Ch}(f_{1,\infty},K)$, respectively.
\end{definition}

\subsection{Packing topological pressure}
The notion of packing topological pressure was studied by Feng and Huang in \cite{FH2012}, inspired by the concept of packing measure in fractal geometry \cite{Mat2004}. 
Let $(X,f_{1,\infty})$ be an NDS on the compact metric space $(X,d)$ and let $K$ be a nonempty subset of $X$. 
For any $\varphi\in C(X,\mathbb{R})$, $s\geq 0$, $N\in \mathbb{N}$, and $\varepsilon>0$, define
\begin{equation}\label{pressuremp}
\mathcal{M}^P(f_{1,\infty},K,\varphi,\varepsilon,s,N)=\sup\left\{\sum_{i}\exp\big(-sn_i+S_{n_i}\varphi(x_i)\big) \right\},
\end{equation}
where the supremum is taken over all finite or countable pairwise disjoint closed Bowen balls $\{\overline{B}_{n_i}(x_i,\varepsilon) \}$ such that $x_i\in K$ and $n_i\geq N$ for all $i$.

Taking the limit as $N\to\infty$, we set 
\[
\mathcal{M}^P(f_{1,\infty},K,\varphi,\varepsilon,s)=\lim_{N\to\infty}\mathcal{M}^P(f_{1,\infty},K,\varphi,\varepsilon,s,N).
\]
Next, define
\[
\mathcal{M}^{\widetilde{P}}(f_{1,\infty},K,\varphi,\varepsilon,s)=\inf\left\{\sum_{i=1}^{\infty}\mathcal{M}^P(f_{1,\infty},K_i,\varphi,\varepsilon,s): K\subseteq \bigcup_{i=1}^{\infty}K_i\right\}.
\]
There exists a unique critical value $s$, denoted by $P^P(f_{1,\infty},K,\varphi,\varepsilon)$, at which $\mathcal{M}^{\widetilde{P}}(f_{1,\infty},K,\varphi,\varepsilon,s)$ jumps from infinity to zero. More precisely,
\begin{align*}
P^P(f_{1,\infty},K,\varphi,\varepsilon)&=\sup\{s:\mathcal{M}^{\widetilde{P}}(f_{1,\infty},K,\varphi,\varepsilon,s)=\infty\}\\
&=\inf\{s: \mathcal{M}^{\widetilde{P}}(f_{1,\infty},K,\varphi,\varepsilon,s)=0\}.
\end{align*}

\begin{definition}
Let $(X,f_{1,\infty})$ be an NDS and $K$ a nonempty subset of $X$. 
For any $\varphi \in C(X,\mathbb{R})$, the \emph{packing topological pressure of $\varphi$ on the set $K$ with respect to $f_{1,\infty}$} is defined by
\[
P^P(f_{1,\infty},K,\varphi)= \lim_{\varepsilon\to 0}P^P(f_{1,\infty},K,\varphi,\varepsilon).
\]
In particular, when $\varphi=0$, the packing topological pressure becomes the packing topological entropy, we denote it by $h_{top}^P(f_{1,\infty},K)$ for consistency \cite{ZZ2023}.
\end{definition}

\begin{remark}
If $S_{n_i}\varphi(x_i)$ in \eqref{Mff} and \eqref{pressuremp} is replaced by $\sup_{y\in B_{n_i}(x_i,\varepsilon)} S_{n_i}\varphi(y)$, then the corresponding values of Pesin topological pressure and packing topological pressure remain unchanged. 
Let $P^{B'}(f_{1,\infty},K,\varphi,\varepsilon)$ and $P^{P'}(f_{1,\infty},K,\varphi,\varepsilon)$ denote the critical values arising from this modification for the Pesin and packing pressures, respectively. By \cite[Proposition 2.3]{ZC2023}, we also have
\begin{align*}
P^{B}(f_{1,\infty},K,\varphi)&=\lim_{\varepsilon\to0}
P^{B'}(f_{1,\infty},K,\varphi,\varepsilon),\\
P^{P}(f_{1,\infty},K,\varphi)&=\lim_{\varepsilon\to0}
P^{P'}(f_{1,\infty},K,\varphi,\varepsilon).
\end{align*}
\end{remark}

\subsection{Measure-theoretic pressure}
Let $\mathcal{B}(X)$ denote the $\sigma$-algebra of Borel subsets of $X$, and let $\mathcal{M}(X)$ denote the collection of all probability measures on the measurable space $(X,\mathcal{B}(X))$. 
Throughout this subsection, we do not assume that a given measure $\mu\in\mathcal{M}(X)$ is invariant under the sequence of transformations $f_{1,\infty}$.
It is well known that $\mathcal{M}(X)$ is compact with respect to the weak$^*$-topology \cite{Wal2000a}.

\begin{definition}
Let $(X,f_{1,\infty})$ be an NDS, $\varphi\in C(X,\mathbb{R})$, and $\mu\in \mathcal{M}(X)$. 
The \emph{upper and lower measure-theoretic local pressures (or Brin-Katok pressures) of $x\in X$ with respect to $\mu$ and $\varphi$} are defined by
\begin{align*}
\overline{P}_{\mu}(f_{1,\infty},x,\varphi)&=\lim_{\varepsilon\to 0}\limsup_{n\to\infty}\frac{-\log \mu(B_n(x,\varepsilon))+S_n\varphi(x)}{n},\\
\underline{P}_{\mu}(f_{1,\infty},x,\varphi)&=\lim_{\varepsilon\to 0}\liminf_{n\to\infty}\frac{-\log \mu(B_n(x,\varepsilon))+S_n\varphi(x)}{n},
\end{align*}
respectively. In particular, when $\varphi=0$, we denote the \emph{upper and lower measure-theoretic local entropy (or Brin-Katok entropies) of $x\in X$ with respect to $\mu$} by $\overline{h}_{\mu}(f_{1,\infty},x)$ and $\underline{h}_{\mu}(f_{1,\infty},x)$, respectively.
\end{definition}

\begin{remark}
The quantities $\overline{h}_{\mu}(f_{1,\infty},x)$ and $\underline{h}_{\mu}(f_{1,\infty},x)$ are also called the local upper and lower $\mu$-measure entropies of $x$ in \cite{Bis2018}.
\end{remark}

\begin{definition} \label{definitiontheoretic}
Let $(X,f_{1,\infty})$ be an NDS, $K\in \mathcal{B}(X)$, $\mu\in \mathcal{M}(X)$, and $\varphi\in C(X,\mathbb{R})$. 
The \emph{upper and lower measure-theoretic (local) pressures of $\mu$ with respect to $\varphi$} over $K$ are defined by
\begin{align*}
\overline{P}_{\mu}(f_{1,\infty},K,\varphi)&=\int_K \overline{P}_{\mu}(f_{1,\infty},x,\varphi) d\mu(x),\\
\underline{P}_{\mu}(f_{1,\infty},K,\varphi)&=\int_K \underline{P}_{\mu}(f_{1,\infty},x,\varphi) d\mu(x),
\end{align*}
respectively. 
In particular, when
$K=X$, we denote $\overline{P}_{\mu}(f_{1,\infty},X,\varphi)$ and $\underline{P}_{\mu}(f_{1,\infty},X,\varphi)$ by $\overline{P}_{\mu}(f_{1,\infty},\varphi)$ and $\underline{P}_{\mu}(f_{1,\infty},\varphi)$ for simplicity, respectively. 
Moreover, when $\varphi=0$, the \emph{upper and lower measure-theoretic (local) entropies of $\mu$} over $K$ are given by
\begin{align*}
\overline{h}_{\mu}(f_{1,\infty},K)&=\overline{P}_{\mu}(f_{1,\infty},K,0)=\int_K \overline{h}_{\mu}(f_{1,\infty},x)d\mu(x),\\
\underline{h}_{\mu}(f_{1,\infty},K)&=\underline{P}_{\mu}(f_{1,\infty},K,0)=\int_K \underline{h}_{\mu}(f_{1,\infty},x)d\mu(x),
\end{align*}
respectively.
\end{definition}

\begin{definition}\label{pressureofmu}
Let $(X,f_{1,\infty})$ be an NDS, $\varphi\in C(X,\mathbb{R})$, and $\mu\in \mathcal{M}(X)$. 
The \emph{Pesin pressure}, \emph{packing pressure}, \emph{lower and upper capacity pressures} of $\mu$ with respect to $\varphi$ are given by
\begin{align*}
P_{\mu}^B(f_{1,\infty},\varphi)&=\lim_{\delta\to0}\inf\{P^B(f_{1,\infty},K,\varphi):\mu(K) \geq 1- \delta \},\\
P_{\mu}^P(f_{1,\infty},\varphi)&= \lim_{\delta\to0}\inf\{P^P(f_{1,\infty},K,\varphi): \mu(K) \geq 1- \delta\},\\
\underline{CP}_{\mu}(f_{1,\infty},\varphi) & =\lim_{\delta\to 0}\inf\{\underline{CP}(f_{1,\infty},K,\varphi):\mu(K)\geq 1-\delta \},\\
\overline{CP}_{\mu}(f_{1,\infty},\varphi) & =\lim_{\delta\to 0}\inf\{\overline{CP}(f_{1,\infty},K,\varphi):\mu(K)\geq 1-\delta \},
\end{align*}
respectively.
\end{definition}

\section{Properties of various pressures}
\begin{proposition}\label{propertiesp}
Let $(X,f_{1,\infty})$ be an NDS and $\varphi \in C(X,\mathbb{R})$. Then the following assertions follow.
\begin{enumerate}
\item[(1)] If $K\subseteq K'\subseteq X$, then
\[
  \mathcal{P}(f_{1,\infty},K,\varphi)\leq \mathcal{P}(f_{1,\infty},K',\varphi),
\]
where $\mathcal{P}$ is chosen from $\{P^B, \underline{CP}, \overline{CP}, P^P\}$.

\item[(2)] If $K = \bigcup_i K_i$, then
\begin{align*}
P^B(f_{1,\infty},K,\varphi) &= \sup_{i\geq 1} P^B(f_{1,\infty},K_i,\varphi) \quad P^P(f_{1,\infty},K,\varphi) = \sup_{i\geq 1} P^P(f_{1,\infty},K_i,\varphi),\\
\overline{CP}(f_{1,\infty},K,\varphi) &\geq \sup_i \overline{CP}(f_{1,\infty},K_i,\varphi),\quad \underline{CP}(f_{1,\infty},K,\varphi)\geq \sup_i \underline{CP}(f_{1,\infty},K_i,\varphi).
\end{align*}
\end{enumerate}
\end{proposition}

\begin{proof}
It suffices to prove the assertions for packing topological pressure, since the corresponding results for the Pesin pressure, upper and lower capacity pressures follow directly by \cite{Pes1997}. 
The first assertion is immediate from the definitions and is therefore omitted.

(2) Suppose that $K= \bigcup_{i\geq 1} K_i$. By (1), we have $P^P(f_{1,\infty},K,\varphi)\geq \sup_{i\geq 1}P^P(f_{1,\infty},K_i,\varphi)$. 
To prove the reverse inequality, for any $\epsilon>0$ and each $i\geq 1$, there exists a collection $\{K_{ij}\}_{j\geq 0}$ such that $K_i\subseteq \bigcup_{j\geq 0}K_{ij}$ and
\[
\sum_{j\geq 0}\mathcal{M}^P(f_{1,\infty},K_{ij},\varphi,\varepsilon,s)- \mathcal{M}^{\widetilde{P}}(f_{1,\infty},K_i,\varphi,\varepsilon,s)\leq \frac{\epsilon}{2^i}.
\]
Consequently,
\[
\sum_{i\geq 1} \sum_{j\geq 0}\mathcal{M}^P(f_{1,\infty},K_{ij},\varphi,\varepsilon,s)\leq \sum_{i\geq 1} \mathcal{M}^{\widetilde{P}}(f_{1,\infty},K_i,\varphi,\varepsilon,s)+\sum_{i\geq 1} \frac{\epsilon}{2^i}.
\]
This yields
\begin{equation}\label{inequality2}
\mathcal{M}^{\widetilde{P}}(f_{1,\infty},K,\varphi,\varepsilon,s)\leq \sum_{i\geq 1}\mathcal{M}^{\widetilde{P}}(f_{1,\infty},K_i,\varphi,\varepsilon,s).
\end{equation}
Now assume that $\sup_{i\geq 1}P^P(f_{1,\infty},K_i,\varphi)<s$. Then for any $\varepsilon>0$ and any $i\geq 1$, $P^P(f_{1,\infty},K_i,\varphi,\varepsilon)<s$, which implies $\mathcal{M}^{\widetilde{P}}(f_{1,\infty},K_i,\varphi,\varepsilon,s)=0$. 
By \eqref{inequality2}, it follows that $\mathcal{M}^{\widetilde{P}}(f_{1,\infty},K,\varphi,\varepsilon,s)=0$, and hence $P^P(f_{1,\infty},K,\varphi,\varepsilon)\leq s$. 
Therefore, $P^P(f_{1,\infty},K,\varphi)\leq \sup_{i\geq 1} P^P(f_{1,\infty},K_i,\varphi)$.
\end{proof}

\begin{proposition}~\label{prop3}
Let $(X,f_{1,\infty})$ be an NDS, $K$ a nonempty subset of $X$, and $\mathcal{P}\in\{P^B, P^P,\underline{CP}, \overline{CP}\}$. 
Then for any $\varphi,\psi \in C(X,\mathbb{R})$, the following assertions hold:
\begin{enumerate}
    \item[(1)] if $\varphi \leq \psi$, then $\mathcal{P}(f_{1,\infty},K,\varphi) \leq \mathcal{P}(f_{1,\infty},K,\psi)$;
    \item[(2)] if $\mathcal{P}$ is finite, then $|\mathcal{P}(f_{1,\infty},K,\varphi)- \mathcal{P}(f_{1,\infty},K,\psi) |\leq \|\varphi- \psi \|$, where $\|\cdot\|$ denotes the supremum norm on the space of continuous functions on $X$;
    \item[(3)]  if $c$ is a real constant, then $\mathcal{P}(f_{1,\infty},K,\varphi +c)=\mathcal{P}(f_{1,\infty},K,\varphi)+c$;
   \item[(4)] 
   if $c \geq 1$, then $\mathcal{P}(f_{1,\infty},K,c\varphi)\leq c\cdot \mathcal{P}(f_{1,\infty},K,\varphi)$; if $c \leq 1$, then $\mathcal{P}(f_{1,\infty},K,c\varphi)\geq c\cdot \mathcal{P}(f_{1,\infty},K,\varphi)$;
   \item[(5)]  $|\mathcal{P}(f_{1,\infty},K,\varphi)|\leq \mathcal{P}(f_{1,\infty},K,|\varphi|)$.
\end{enumerate}
\end{proposition}

\begin{proof}
The proofs of these inequalities are essentially analogous, therefore, for simplicity, we present the argument only for the Pesin topological pressure.

(1) The first assertion follows directly from the definitions.

(2) For any $x\in X$ and $\varphi, \psi\in C(X,\mathbb{R})$, we have
\[
  \frac{1}{n} \sum_{j=0}^{n-1} \left|\varphi(f_1^j(x))- \psi(f_1^j(x))\right| \leq \|\varphi-\psi\|.
\]
Consequently,
\[
S_n\psi(x)-n\|\varphi-\psi\|\leq S_n\varphi(x)\leq S_n\psi(x)+n\|\varphi-\psi\|.
\]
Therefore,
\[
\sum_{i}e^{-sn_i-n_i\|\varphi-\psi\|+S_{n_i}\psi(x_i)}\leq \sum_{i}e^{-sn_i+S_{n_i}\varphi(x_i)}\leq \sum_{i}e^{-sn_i+n_i\|\varphi-\psi\|+S_{n_i}\psi(x_i)}.
\]
It follows that
\[
M(f_{1,\infty},K,\psi,\varepsilon,s+\|\varphi-\psi\|,N)\leq M(f_{1,\infty},K,\varphi,\varepsilon,s,N)\leq M(f_{1,\infty},K,\psi,\varepsilon,s-\|\varphi-\psi\|,N).
\]
Therefore,
\[
P^B(f_{1,\infty},K,\psi)-\|\varphi-\psi\|\leq P^B(f_{1,\infty},K,\varphi)\leq P^B(f_{1,\infty},K,\psi)+\|\varphi-\psi\|.
\]
Equivalently, $|P^B(f_{1,\infty},K,\varphi)-P^B(f_{1,\infty},K,\psi)|\leq \|\varphi-\psi\|$.

(3) This assertion follows from the observation that $\exp(-sn_i+S_{n_i}\varphi(x_i)+n_i c)=\exp(n_i(c-s)+S_{n_i}\varphi(x_i))$, which completes the proof.

(4) Recall that if $\{a_i\}_i$ is a family of positive numbers with $\sum_{i}a_i=1$, then $\sum_i a_i^c\leq 1$ for $c \geq 1$, and $\sum_i a_i^c \geq 1$ for $c \leq 1$. 
Thus, for $c \geq 1$,
\begin{align*}
   \sum_i e^{-sn_i+S_{n_i}(c\varphi)(x_i)}&=\sum_i e^{-sn_i+ c\cdot S_{n_i}\varphi(x_i)}\\
    &=\sum_i \left(e^{-\frac{s}{c} n_i+S_{n_i}\varphi(x_i)}\right)^c\\
    &\leq \left(\sum_i e^{-\frac{s}{c}n_i+S_{n_i}\varphi(x_i)}\right)^c.
\end{align*}
Consequently, 
\[
   M(f_{1,\infty},K,c\varphi,\varepsilon,s,N) \leq M(f_{1,\infty},K,\varphi,\varepsilon,s/c,N)^c
\]
and letting $N\to\infty$ yields $P^B(f_{1,\infty},K,c \varphi)\leq c\cdot P^B(f_{1,\infty},K,\varphi)$. 
Similarly, if $c \leq 1$, then $P^B(f_{1,\infty},K,c \varphi) \geq c\cdot P^B(f_{1,\infty},K,\varphi)$.

(5) Since $-|\varphi| \leq \varphi \leq |\varphi|$, we have
\[
  P^B(f_{1,\infty},K,-|\varphi|)\leq P^B(f_{1,\infty},K,\varphi)\leq P^B(f_{1,\infty},K,|\varphi|).
\]
By part (4), it follows that 
\[
  P^B(f_{1,\infty},K,-|\varphi|) \geq -P^B(f_{1,\infty},K,|\varphi|).
\] 
Therefore, $|P^B(f_{1,\infty},K,\varphi)|\leq P^B(f_{1,\infty},K,|\varphi|)$.
\end{proof}
\medskip

We denote by
\begin{equation}\label{Lambda}
    \Lambda_{\varepsilon,N}(f_{1,\infty},K,\varphi)=\inf\left\{\sum_{i}\exp(S_N\varphi(x_i)) \right\},
\end{equation}
where the infimum is taken over all finite or countable collections of $\{B_N(x_i,\varepsilon)\}_i$ such that $x_i\in X$ and $K\subseteq \bigcup_i B_N(x_i,\varepsilon)$.

\begin{proposition}
Let $(X,f_{1,\infty})$ be an NDS, $K$ a nonempty subset of $X$, and $\varphi\in C(X,\mathbb{R})$. 
By the Carath{\'e}odory-Pesin structure, then
\begin{align*}
\overline{CP}(f_{1,\infty},K,\varphi)&=\lim_{\varepsilon\to 0}\limsup_{N\to\infty}\frac{1}{N}\log \Lambda_{\varepsilon,N}(f_{1,\infty},K,\varphi),\\
\underline{CP}(f_{1,\infty},K,\varphi)&=\lim_{\varepsilon\to 0}\liminf_{N\to\infty}\frac{1}{N}\log \Lambda_{\varepsilon,N}(f_{1,\infty},K,\varphi).
\end{align*}
\end{proposition}

\begin{proof}
We only need to prove the statement for the upper capacity topological pressure, namely,
\[
  \overline{CP}(f_{1,\infty},K,\varphi,\varepsilon) =\limsup_{N\to\infty}\frac{1}{N}\log \Lambda_{\varepsilon,N}(f_{1,\infty},K,\varphi).
\]
The corresponding result for the lower capacity topological pressure can be proved in an analogous manner.

Let
\[
  s=\overline{CP}(f_{1,\infty},K,\varphi,\varepsilon) \quad \text{and}\quad t=\limsup_{N\to\infty}\frac{1}{N}\log \Lambda_{\varepsilon,N}(f_{1,\infty},K,\varphi).
\]
It follows directly from the definition that
\begin{equation}\label{rlambda}
R(f_{1,\infty},K,\varphi,\varepsilon,s,N)=e^{-sN}\cdot \Lambda_{\varepsilon,N}(f_{1,\infty},K,\varphi).
\end{equation}
Fix a small number $\gamma>0$. By the definition of $\overline{CP}(f_{1,\infty},K,\varphi,\varepsilon)$, there exists a sequence of positive integers $\{N_n\}$ such that
\[
\overline{r}(f_{1,\infty},K,\varphi,\varepsilon,s+\gamma)=\lim_{n\to\infty}R(f_{1,\infty},K,\varphi,\varepsilon,s+\gamma,N_n)=0.
\]
Consequently, for sufficiently large $n$, $R(f_{1,\infty},K,\varphi,\varepsilon,s+\gamma,N_n) \leq 1$. 
Using \eqref{rlambda}, we obtain
\[
\Lambda_{\varepsilon,N_n}(f_{1,\infty},K,\varphi)\cdot (e^{-N_n})^{s+\gamma}\leq 1.
\]
Since $e^{-N_n}<1$ for large $n$, this implies
\[
\frac{1}{N_n}\log \Lambda_{\varepsilon,N_n}(f_{1,\infty},K,\varphi)\leq s+\gamma.
\]
Therefore,
\begin{equation}\label{left}
t \leq \limsup_{n\to\infty} \frac{1}{N_n}\log \Lambda_{\varepsilon,N_n}(f_{1,\infty},K,\varphi)\leq s+\gamma.
\end{equation}

\medskip

Next, choose another sequence $\{N_n'\}$ such that
\[
\lim_{n\to \infty}R(f_{1,\infty},K,\varphi,\varepsilon,s-\gamma,N_n')\geq \overline{r}(f_{1,\infty},K,\varphi,\varepsilon,s-\gamma)=\infty.
\]
Then there exists $n_0$ such that, for all $n \geq n_0$, $R(f_{1,\infty},K,\varphi,\varepsilon,s-\gamma,N_n')\geq 1$. 
By \eqref{rlambda}, this yields
\[
\Lambda_{\varepsilon,N_n'}(f_{1,\infty},K,\varphi)\cdot (e^{-N_n'})^{s-\gamma}\geq 1,
\]
or equivalently,
\[
\frac{1}{N_n'}\log \Lambda_{\varepsilon,N_n'}(f_{1,\infty},K,\varphi)\geq s-\gamma.
\]
Hence
\begin{equation}\label{right}
t = \limsup_{n\to\infty} \frac{1}{N_n'}\log \Lambda_{\varepsilon,N_n'}(f_{1,\infty},K,\varphi)\geq s-\gamma.
\end{equation}
Combining \eqref{left} and \eqref{right}, and using the arbitrariness of $\gamma$, we conclude that $s=t$. That is, $\overline{CP}(f_{1,\infty},K,\varphi,\varepsilon) =\limsup_{N\to\infty}\frac{1}{N}\log \Lambda_{\varepsilon,N}(f_{1,\infty},K,\varphi)$.
\end{proof}

We have introduced several topological pressures on subsets so far, the immediate \Cref{relationship} describes the relationship among the Pesin pressure, packing pressure, lower and upper capacity pressures, as well as the classical pressure for nonautonomous dynamical systems.

\begin{theorem}\label{relationship}
Let $(X,f_{1,\infty})$ be an NDS, $K$ a nonempty subset of $X$, and $\varphi \in C(X,\mathbb{R})$. Then
\begin{enumerate}
\item[(1)] $P^B(f_{1,\infty},K,\varphi)\leq \underline{CP}(f_{1,\infty},K,\varphi)\leq \overline{CP}(f_{1,\infty},K,\varphi)$;

\item[(2)] $P^B(f_{1,\infty},K,\varphi)\leq P^P(f_{1,\infty},K,\varphi)\leq P(f_{1,\infty},K,\varphi)$;

\item[(3)] $P^P(f_{1,\infty},K,\varphi)\leq \overline{CP}(f_{1,\infty},K,\varphi)$;

\item[(4)] $\overline{CP}(f_{1,\infty},K,\varphi)=P(f_{1,\infty},K,\varphi)$.
\end{enumerate}
\end{theorem}

\begin{proof}
We assume throughout that the pressures under consideration are not equal to $- \infty$; otherwise, there is nothing to prove.

(1) The first assertion follows directly from the definition of Pesin topological pressure.

(2) We first prove the inequality $P^B(f_{1,\infty},K,\varphi)\leq P^P(f_{1,\infty},K,\varphi)$. 
Following the approach of \cite[Theorem 5.12]{Mat2004}, fix a real number $s$ such that $-\infty<s<P^B(f_{1,\infty},K,\varphi)$. 
For any fixed $n$ and $\varepsilon$, let $\gamma=R_n(K,\varepsilon)$ denote the largest number so that there exists a disjoint family $\{\overline{B}_{n_i}(x_i,\varepsilon)\}_{i=1}^{\gamma}$ with $x_i\in K$. 
Such a family exists since $X$ is compact.
Then, for any $\delta>0$,
\[
K\subseteq \bigcup_{i=1}^{\gamma} \overline{B}_{n_i}(x_i,2\varepsilon+\delta).
\]
By \eqref{Mff}, for any $s>0$ and any $N\in \mathbb{N}$,
\begin{align*}
M(f_{1,\infty},K,\varphi,2\varepsilon+\delta,s,N)&\leq e^{-sN}\sum_{i=1}^{\gamma}\exp(S_{n_i}\varphi(x_i))\\
&\leq \mathcal{M}^P(f_{1,\infty},K,\varphi,\varepsilon,s,N).
\end{align*}
Letting $N\to\infty$, we obtain
\begin{equation}\label{mmcM}
M(f_{1,\infty},K,\varphi,2\varepsilon+\delta,s)\leq \mathcal{M}^{\widetilde{P}}(f_{1,\infty},K,\varphi,\varepsilon,s).
\end{equation}
Since $-\infty<s<P^B(f_{1,\infty},K,\varphi)$, it follows from the definition of $P^B(f_{1,\infty},K,\varphi)$ that $M(f_{1,\infty},K,\varphi,2\varepsilon+\delta)=+\infty$. 
Hence, for sufficiently small $\varepsilon$ and $\delta$, $M(f_{1,\infty},K,\varphi,2\varepsilon+\delta,s)\geq 1$,
and consequently, $\mathcal{M}^{\widetilde{P}}(f_{1,\infty},K,\varphi,2\varepsilon+\delta,s)\geq 1$. 
Therefore, $s\leq P^P(f_{1,\infty},K,\varphi,\varepsilon)$ for all sufficiently small $\varepsilon$. 
Letting $\varepsilon\to 0$, we conclude that $s \leq P^P(f_{1,\infty},K,\varphi)$.
Since $s<P^B(f_{1,\infty},K,\varphi)$ is arbitrary, this yields $P^B(f_{1,\infty},K,\varphi) \leq P^P(f_{1,\infty},K,\varphi)$.

We now prove the second inequality $P^P(f_{1,\infty},K,\varphi)\leq P(f_{1,\infty},K,\varphi)$. 
Following the method of \cite{FH2012,ZC2023}, choose real numbers $t,s$ such that 
\[
  -\infty< t<s<P^P(f_{1,\infty},K,\varphi).
\] 
Then there exists $\delta>0$ such that, for any $\varepsilon\in (0,\delta)$, 
\[
  P^P(f_{1,\infty},K,\varphi,\varepsilon)>s \quad \text{and}\quad \mathcal{M}^P(f_{1,\infty},K,\varphi,\varepsilon)\geq \mathcal{M}^{\widetilde{P}}(f_{1,\infty},K,\varphi,\varepsilon)=\infty.
\]
For any $N>0$ there exists a countable pairwise disjoint family $\{\overline{B}_{n_i}(x_i,\varepsilon)\}$ such that $x_i\in K$, $n_i\geq N$ for all $i$ and $\sum_i\exp(-sn_i+S_{n_i}\varphi(x_i))>1$. 
For each $j$, let $m_j=\{x_i:n_i=j\}$. 
Then
\[
\sum_{j=N}^{\infty}\sum_{x_i\in m_j}\exp(-js)\cdot\exp(S_j\varphi(x_i))>1.
\]
We claim that there exists $j\geq N$ such that
\begin{equation}\label{ejt}
\sum_{x_i\in m_j}\exp(S_j\varphi(x_i))\geq e^{jt}(1-e^{t-s}).
\end{equation}
Indeed, if this is not the case, then
\begin{align*}
\sum_{j=N}^{\infty}\sum_{x_i\in m_j}\exp(-js)\cdot\exp(S_j\varphi(x_i)) &\leq (1-e^{t-s})\cdot \sum_{j=N}^{\infty}e^{-js}\cdot e^{jt}\\
&=(1-e^{t-s})\cdot e^{(t-s)N}\sum_{j=0}^{\infty}\big(e^{t-s}\big)^j\\
&\leq e^{(t-s)N}<1,
\end{align*}
which is a contradiction.

Now let $E$ be a $(j,\varepsilon)$-separated set for $K$. 
By the construction of the family $\{\overline{B}_{n_i}(x_i,\varepsilon)\}$,
\[
\sum_{x_i\in E}\exp(S_j\varphi(x_i))\geq \sum_{x_i\in m_j}\exp(S_j(\varphi(x_i)))\geq e^{jt}\cdot(1-e^{t-s}).
\]
Therefore,
\begin{align*}
P(f_{1,\infty},K,\varphi,\varepsilon)&=\limsup_{j\to\infty}\frac{1}{j}\log P_j(f_{1,\infty},K,\varphi,\varepsilon)\\
&\geq \limsup_{j\to\infty}\frac{1}{j}\log \sum_{x_i\in E}\exp(S_j\varphi(x_i))\\
&\geq \limsup_{j\to\infty}\frac{1}{j}\log\big[e^{jt}\cdot(1-e^{t-s})\big] \quad\text{(by \eqref{ejt}) }\\
&\geq \limsup_{j\to\infty} \frac{1}{j}\big[jt+\log(1-e^{t-s})\big]\\
&\geq t.
\end{align*}
Letting $\varepsilon\to 0$, we obtain
$P(f_{1,\infty},K,\varphi)\geq t$. 
Since $t<s<P^P(f_{1,\infty},K,\varphi)$ is arbitrary, this implies that $P^P(f_{1,\infty},K,\varphi)\leq P(f_{1,\infty},K,\varphi)$.

(3) We assume that $-\infty<t<s<P^P(f_{1,\infty},K,\varphi)$. 
By part (2), we also have \eqref{ejt}. 
Fix a collection of Bowen balls $\{B_j(y_i,\frac{\varepsilon}{2})\}_{i\in \mathcal{I}}$ such that $K\subseteq \bigcup_{i\in \mathcal{I}} B_j(y_i,\frac{\varepsilon}{2})$. 
For any two distinct points $x_1,x_2\in m_j$, there exist $y_1,y_2$ such that $x_1\in B_j(y_1,\frac{\varepsilon}{2})$ and $x_2\in B_j(y_2,\frac{\varepsilon}{2})$. 
Consequently, by the definition of $R(f_{1,\infty},K,\varphi,\frac{\varepsilon}{2},t,j)$, we have
\begin{align*}
R(f_{1,\infty},K,\varphi,\frac{\varepsilon}{2},t,j)&=\inf\{\sum_i\exp(-jt+S_j\varphi(x_i))\}\\
&\geq \sum_{x_i\in m_j} e^{-jt}e^{S_j\varphi(x_i)}\\
&=e^{-jt}\sum_{x_i\in m_j}e^{S_j\varphi(x_i)}\\
&\geq e^{-jt}\cdot e^{jt}\cdot (1-e^{t-s})\quad \text{(by \eqref{ejt})}\\
&=1-e^{t-s}>0.
\end{align*}
Therefore, 
\[
\overline{r}(f_{1,\infty},K,\varphi,\frac{\varepsilon}{2},t)=\limsup_{j\to\infty}R(f_{1,\infty},K,\varphi,\frac{\varepsilon}{2},t,j) \geq 1-e^{t-s}> 0,
\]
which implies $\overline{CP}(f_{1,\infty},K,\varphi,\frac{\varepsilon}{2})\geq t$. 
Since $t<s<P^P(f_{1,\infty},K,\varphi)$ is arbitrary, we conclude that $P^P(f_{1,\infty},K,\varphi)\leq \overline{CP}(f_{1,\infty},K,\varphi)$.

(4) By \eqref{Lambda}, the infimum of $\Lambda_{\varepsilon,N}(f_{1,\infty},K,\varphi)$ is taken over all finite or countable collections of open Bowen balls $\{B_N(x_i,\varepsilon)\}$ that cover $K$. 
Each such ball $B_N(x_i,\varepsilon)$ is also a $(N,\varepsilon)$-spanning set for $K$. Consequently, the minimal numbers of such open balls coincides with $r_N(f_{1,\infty},K,\varepsilon)$, the smallest cardinality of any $(N,\varepsilon)$-spanning set for $K$. 
Therefore, by the definitions of the classical topological pressure and the upper capacity pressure, we have $Q_N(f_{1,\infty},K,\varphi,\varepsilon)=\Lambda_{\varepsilon,N}(f_{1,\infty},K,\varphi)$. 
It follows that $\overline{CP}(f_{1,\infty},K,\varphi)=P(f_{1,\infty},K,\varphi)$.
\end{proof}

\begin{remark}
When $(X,f_{1,\infty})$ reduces to the classical dynamical system $(X,f)$ and $K$ is a compact $f$-invariant subset of $X$, Feng and Huang~\cite{FH2012} showed that
\[
  P^B(f,K,\varphi)=P^P(f,K,\varphi)=\underline{CP}(f,K,\varphi)=\overline{CP}(f,K,\varphi)=P(f,K,\varphi).
\] 
In contrast, due to the intrinsic characteristics of NDS, such an equality between the lower and upper capacity pressures cannot hold, in general, even when $K$ is $f_{1,\infty}$-invariant.
\end{remark}

\begin{corollary}\label{forentropy}
Let $(X,f_{1,\infty})$ be an NDS, and $K$ a nonempty subset of $X$. Then
\begin{enumerate}
  \item[(1)] $h_{top}^B(f_{1,\infty},K)\leq \underline{Ch}(f_{1,\infty},K)\leq \overline{Ch}(f_{1,\infty},K)$;
  \item[(2)] $h_{top}^B(f_{1,\infty},K)\leq h_{top}^{P}(f_{1,\infty},K)\leq h(f_{1,\infty},K)$;
  \item[(3)] $h_{top}^P(f_{1,\infty},K)\leq \overline{Ch}(f_{1,\infty},K)$;
  \item[(4)] $\overline{Ch}(f_{1,\infty},K)=h(f_{1,\infty},K)$.
\end{enumerate}
\end{corollary}

\begin{lemma}\label{lipschitz}
Assume that $P(f_{1,\infty},K,\varphi)<+\infty$ for any $\varphi\in C(X,\mathbb{R})$. 
Then for any $\varphi,\psi\in C(X,\mathbb{R})$,
\begin{align*}
|P^B(f_{1,\infty},K,\varphi)-P^B(f_{1,\infty},K,\psi)| &\leq \|\varphi-\psi\|,\\
|P^P(f_{1,\infty},K,\varphi)-P^P(f_{1,\infty},K,\psi)| &\leq \|\varphi-\psi\|,
\end{align*}
where $\|\varphi\|:=\sup_{x\in X}|\varphi(x)|$.
\end{lemma}

\begin{proof}
This result can follows from \Cref{relationship} and \cite[Theorem 9.7]{Wal2000a}, see also \Cref{prop3}, so we omit it.
\end{proof}

\begin{theorem}[Bowen's equation]
If $h(f_{1,\infty})<\infty$, then the functions $t\mapsto P^B(f_{1,\infty},K,t\varphi)$ and $t\mapsto P^P(f_{1,\infty},K,t\varphi)$ are strictly decreasing and Lipschitz continuous. 
Furthermore, there exist unique roots $t_B$ and $t_P$ of the respective equations
\[
P^B(f_{1,\infty},K,t\varphi)=0 \quad \text{and}\quad P^P(f_{1,\infty},K,t\varphi)=0,
\]
which satisfy
\[
0\leq t_B\leq t_P<\infty.
\]
\end{theorem}

\begin{proof}
The proof is inspired by the argument in \cite{Bar1996}. 
Note that $X$ is compact and $\varphi\in C(X,\mathbb{R})$, so that $\|\varphi\|<\infty$. 
It is straightforward to verify that, for $t\in\mathbb{R}$,
\begin{align*}
h_{top}^B(f_{1,\infty},K)-\|t\varphi\|&\leq P^B(f_{1,\infty},K,t\varphi)\leq h_{top}^B(f_{1,\infty},K)+\|t\varphi\|,\\
h_{top}^P(f_{1,\infty},K)-\|t\varphi\|&\leq P^P(f_{1,\infty},K,t\varphi)\leq h_{top}^P(f_{1,\infty},K)+\|t\varphi\|.
\end{align*}
Furthermore, since $h(f_{1,\infty})$ is finite and $h(f_{1,\infty},K)\leq h(f_{1,\infty})$ for any subset $K$ of $X$, it follows from part (2) of \Cref{forentropy} that these two functions are finite.

By \Cref{lipschitz}, we conclude that these pressure functions are Lipschitz continuous. 
If $t'\geq t$, then
\begin{align}
-(t'-t)\|\varphi\|\leq P^B(f_{1,\infty},K,t\varphi)-P^B(f_{1,\infty},K,t'\varphi)&\leq (t'-t)\|\varphi\|,\label{pesin}\\
-(t'-t)\|\varphi\|\leq P^P(f_{1,\infty},K,t\varphi)-P^P(f_{1,\infty},K,t'\varphi)&\leq (t'-t)\|\varphi\|.\label{packing}
\end{align}
Thus, these inequalities imply that $P^B(f_{1,\infty},K,t\varphi)$ and $P^P(f_{1,\infty},K,t\varphi)$ are strictly decreasing, and hence there exist unique roots $t_B$ and $t_P$ of the respective equations $P^B(f_{1,\infty},K,t\varphi)=0$ and $P^P(f_{1,\infty},K,t\varphi)=0$. 
It is clear by \Cref{relationship} that $t_B\leq t_P$.

Now, observe that $P^B(f_{1,\infty},K,0)=h_{top}^B(f_{1,\infty},K)\geq 0$. Setting $t=0$ and $P^B(f_{1,\infty},K,t'\varphi)=0$ in \eqref{pesin}, we obtain
\[
  -t_B \|\varphi\|\leq h_{top}^B(f_{1,\infty},K)\leq t_B\|\varphi\|.
\]
Thus, $t_B\geq 0$. 
Similarly, observe that $P^P(f_{1,\infty},K,0)=h_{top}^P(f_{1,\infty},K)\leq h(f_{1,\infty})<\infty$. Setting $t=0$ and $P^P(f_{1,\infty},K,t'\varphi)=0$ in \eqref{packing} yields
\[
  -t_P \|\varphi\|\leq h_{top}^P(f_{1,\infty},K)\leq t_P\|\varphi\|.
\]
Thus, we conclude that $t_P<\infty$.
\end{proof}

\begin{proposition}[Inverse variational principle]\label{ppmeasure}
Let $(X,f_{1,\infty})$ be an NDS, $K$ a nonempty subset of $X$, and $\varphi\in C(X,\mathbb{R})$. 
The Pesin pressure and the packing pressure of measure $\mu$ with respect to $\varphi$ can also be defined as
\begin{align*}
P_{\mu}^B(f_{1,\infty},\varphi) &=\inf\{P^B(f_{1,\infty},K,\varphi):\mu(K) = 1\},\\
P_{\mu}^P(f_{1,\infty},\varphi) &=\inf\{P^P(f_{1,\infty},K,\varphi):\mu(K) = 1\}.
\end{align*}
However,
\begin{align*}
\underline{CP}_{\mu}(f_{1,\infty},\varphi) &\leq \inf\{\underline{CP}(f_{1,\infty},K,\varphi):\mu(K) = 1\},\\
\overline{CP}_{\mu}(f_{1,\infty},\varphi) &\leq \inf\{\overline{CP}(f_{1,\infty},K,\varphi):\mu(K) = 1\},
\end{align*}
and the strict inequalities may occur.
\end{proposition}

\begin{proof}
Let
\[
  P_{\mu}^{B'}(f_{1,\infty},\varphi)=\inf\{P^B(f_{1,\infty},K,\varphi):\mu(K) = 1 \}.
\]
It is immediate that
\[
{P}_{\mu}^B(f_{1,\infty},\varphi) \leq P_{\mu}^{B'}(f_{1,\infty},\varphi).
\]
To prove the reverse inequality, one can construct an increasing sequence $\{K_n\}$ such that $\mu(K_n)\geq 1-\frac{1}{n}$ for each $n$ (with $\delta=\frac{1}{n}<1$) and $P_{\mu}^{B'}(f_{1,\infty},\varphi)=\lim_{n\to\infty} P^B(f_{1,\infty},K_n,\varphi)$. 
Consequently, $P_{\mu}^{B'}(f_{1,\infty},\varphi)=\sup_n P^B(f_{1,\infty},K_n,\varphi)$. 
Let $K=\bigcup_n K_n$. By \Cref{propertiesp}, we have
\[
P_{\mu}^{B}(f_{1,\infty},\varphi)\leq P^B(f_{1,\infty},K,\varphi) = \sup_{n} P^B(f_{1,\infty},K_n,\varphi)=P_{\mu}^{B'}(f_{1,\infty},\varphi),
\]
which yields the desired inequality.

An analogous argument applies to the equivalent formulation of packing pressure.
For the lower and upper capacity pressures of $\mu$ with respect to $\varphi$, the corresponding inequalities follow directly from the definitions. Examples showing that these inequalities may be strict are given in \cite[Example 7.1]{Pes1997}.
\end{proof}

He et al. \cite{HLZ2004} introduced an alternative notion of measure-theoretic pressure for dynamical systems. 
We extend this concept to the NDS $(X,f_{1,\infty})$. 
Let $\mu\in \mathcal{M}(X)$ and $\varphi\in C(X,\mathbb{R})$. 
For $\varepsilon>0$ and $n\in \mathbb{N}$, define
\begin{align*}
P_{\mu}^*(f_{1,\infty},\varphi,\varepsilon,n)=\lim_{\delta\to 0}\inf\Big\{\sum_{x\in F}\exp(S_n\varphi(x)) &: F \text{ is an }(n,\varepsilon)\text{-spanning set of }\\
&\text{a set of } \mu\text{-measure greater than or equal to }1-\delta \Big\}.
\end{align*}
Set
\[
  P_{\mu}^*(f_{1,\infty},\varphi,\varepsilon)=\limsup_{n\to\infty}\frac{1}{n}\log P_{\mu}^*(f_{1,\infty},\varphi,\varepsilon,n),
\]
and define
\[
  P_{\mu}^*(f_{1,\infty},\varphi)=\lim_{\varepsilon\to 0}P_{\mu}^*(f_{1,\infty},\varphi,\varepsilon).
\]

Observe that $P_{\mu}^*(f_{1,\infty},\varphi,\varepsilon,n)$ admits the following equivalent representation:
\[
  P_{\mu}^*(f_{1,\infty},\varphi,\varepsilon,n)=\lim_{\delta\to 0}\inf\{Q_n(f_{1,\infty},K,\varphi,\varepsilon): \mu(K)\geq 1-\delta\}.
\]
Consequently, by \eqref{Qn}, we obtain
\[
  P_{\mu}^*(f_{1,\infty},\varphi,\varepsilon)=\lim_{\delta\to 0}\inf\{Q(f_{1,\infty},K,\varphi,\varepsilon): \mu(K)\geq 1-\delta\}.
\]

\begin{proposition}
Let $(X,f_{1,\infty})$ be an NDS, and $\varphi\in C(X,\mathbb{R})$. 
Then, for any $\mu\in \mathcal{M}(X)$,
\begin{align*}
P_{\mu}^B(f_{1,\infty},\varphi)&\leq \underline{CP}_{\mu}(f_{1,\infty},\varphi)\leq \overline{CP}_{\mu}(f_{1,\infty},\varphi),\\
P_{\mu}^B(f_{1,\infty},\varphi)
&\leq P_{\mu}^P(f_{1,\infty},\varphi)\leq \overline{CP}_{\mu}(f_{1,\infty},\varphi),\\
P_{\mu}^*(f_{1,\infty},\varphi)&=\overline{CP}_{\mu}(f_{1,\infty},\varphi)\leq P(f_{1,\infty},X,\varphi).
\end{align*}
\end{proposition}

\begin{proof}
These inequalities follow directly from Theorem \ref{relationship}.
\end{proof}

The following pressure distribution principle for Pesin topological pressure, which resembles the mass distribution principle in fractal geometry, extends the work on the entropy distribution principle \cite[Theorem 3.6]{TV2003} and provides a method for obtaining the lower bound of Pesin topological pressure.

\begin{theorem}[Pressure distribution principle for Pesin topological pressure] \label{distribution}
Let $(X,f_{1,\infty})$ be an NDS, $\mathcal{M}(X)$ the set of all Borel probability measures on $X$, $K$ a nonempty subset of $X$, and $\varphi$ a continuous function on $X$.
Suppose that there exist $\varepsilon>0$ and $s\geq 0$ such that one can find a sequence of Borel probability measures $\{\mu_k\}\subset \mathcal{M}(X)$, a constant $C>0$ and an integer $n\in \mathbb{N}$ satisfying the following conditions:
\begin{enumerate}
\item[(1)] $\mu_k(K)>0$ for all $k$;
\item[(2)] $\limsup_{k \to\infty} \mu_k(B_n(x,\varepsilon))\leq C \exp(-ns+S_n\varphi(x))$ for each Bowen ball $B_n(x,\varepsilon)$ such that $B_n(x,\varepsilon)\cap K\neq\emptyset$;
\item[(3)] the sequence $\{\mu_k\}$ admits at least one limit measure $\mu$ satisfying $\mu(K)>0$.
\end{enumerate}
Then $P^B(f_{1,\infty},K,\varphi)\geq s$.
\end{theorem}

\begin{proof}
To prove this theorem, it suffices to show that $M(f_{1,\infty},K,\varphi,\varepsilon,s)>0$ for all sufficiently small $\varepsilon>0$ satisfying the above conditions.
To this end, let $\Gamma=\{B_{n_i}(x_i,\varepsilon)\}_i$ be a cover of $K$. Without loss of generality, we may assume that each element of $\Gamma$ satisfies $B_{n_i}(x_i,\varepsilon)\cap K \neq\emptyset$. 
Then
\begin{align*}
\sum_i \exp{(-sn_i+S_{n_i}\varphi(x_i))}&\geq \frac{1}{C}\lim_{k\to\infty}\sum_i\mu_k(B_{n_i}(x_i,\varepsilon))\\
&\geq \frac{1}{C}\lim_{k\to\infty}\mu_k\left(\bigcup_i B_{n_i}(x_i,\varepsilon) \right)\\
&\geq \frac{1}{C}\lim_{k\to\infty}\mu_k(K)\\
&\geq \frac{1}{C}\mu(K)>0.
\end{align*}
This implies that $M(f_{1,\infty},K,\varphi,\varepsilon,s,N)>0$, and hence $P^B(f_{1,\infty},K,\varphi)\geq s$.
\end{proof}

\begin{corollary}[Pressure distribution principle for packing topological pressure]
Let $(X,f_{1,\infty})$ be an NDS.
Then under the same conditions as in \Cref{distribution}, we also have the pressure distribution principle for packing topological pressure.
\end{corollary}

\begin{proof}
This result follows directly from \eqref{mmcM}, \Cref{relationship}, and \Cref{distribution}.
\end{proof}

\begin{remark}
When $\varphi=0$, a similar result for the Bowen topological entropy was discussed by Lin et al. \cite{LMYT2024} for the shrinking target problem of NDS.
\end{remark}

\section{Variational principle for packing topological pressure of subsets}
\begin{lemma}[Vitali covering lemma]\cite{Mat2004} \label{vitali}
Let $(X,d)$ be a compact metric space, and let $\mathcal{B}=\{B(x_i,r_i)\}_{i\in \mathcal{I}}$ be a family of closed or open balls in $X$. 
Then there exists a finite or countable subfamily $\mathcal{B}'=\{B(x_i,r_i)\}_{i\in \mathcal{I}'}$ of pairwise disjoint balls from $\mathcal{B}$ such that
\[
\bigcup_{B\in \mathcal{B}}B\subseteq \bigcup_{i\in \mathcal{I}'}B(x_i,5r_i).
\]
\end{lemma}

\begin{theorem}[Billingsley type theorem for packing topological pressure]\label{Billingsley}
Let $(X,f_{1,\infty})$ be an NDS, $K$ a nonempty subset of $X$, $\varphi\in C(X,\mathbb{R})$, $\mu\in \mathcal{M}(X)$, and $s\in \mathbb{R}$.
Then the following assertions hold:
\begin{enumerate}
 \item[(1)] If $\overline{P}_{\mu}(f_{1,\infty},x,\varphi)\leq s$ for all $x\in K$, then $P^P(f_{1,\infty},K,\varphi)\leq s$;

 \item[(2)] If $\overline{P}_{\mu}(f_{1,\infty},x,\varphi)\geq s$ for all $x\in K$ and $\mu(K)>0$, then $P^P(f_{1,\infty},K,\varphi)\geq s$.
 \end{enumerate}
\end{theorem}

\begin{proof}
(1) Recall that 
\[
  \overline{P}_{\mu}(f_{1,\infty},x,\varphi)=\lim_{\varepsilon\to0}\limsup_{n\to\infty}\frac{1}{n}(-\log\mu(B_n(x,\varepsilon))+S_n\varphi(x)).
\] 
Fix $r>0$, and for each $m\in \mathbb{N}$ define
\[
K_m=\left\{x\in K: \limsup_{n\to\infty}\frac{1}{n}\big(-\log\mu(B_n(x,\varepsilon))+S_n\varphi(x)\big)<s+r\, \text{for all}\ \varepsilon\in \left(0,1/m\right)\right\}.
\]
Since $\overline{P}_{\mu}(f_{1,\infty},x,\varphi)\leq s$ for all $x\in K$, we have $K=\bigcup_{m\geq 1}K_m$.
Following a method of \cite{ZC2023}, fix $m$ and $\varepsilon\in (0,1/m)$. 
For each $x\in K_m$, there exists $N\in \mathbb{N}$ such that
\[
\mu\left(B_{n}(x,\varepsilon)\right)\geq e^{-n(s+r)+S_{n}\varphi(x)},\quad \forall n\geq N.
\]
For any $N\geq 1$, define
\[
  K_{m,N}=\{x \in K_m: \mu\left( B_{n}(x,\varepsilon)\right)\geq e^{-n(s+r)+S_{n}\varphi(x)},\quad \forall n\geq N\}.
\]
Then $K_m=\bigcup_{N\geq 1} K_{m,N}$.
Fix $N$ and let $N^* \geq N$. 
Consider any finite or countable disjoint family $\{\overline{B}_{n_i}(x_i,\varepsilon)\}_i$ with $x_i\in K_{m,N}$ and $n_i\geq N^*$. 
Then
\[
  \sum_i e^{-n_i(s+2r)+S_{n_i}\varphi(x_i)}=\sum_i e^{-n_i(s+r)+S_{n_i}\varphi(x_i)} \cdot e^{-n_i r}\leq e^{-r N^*}\sum_i\mu(B_{n_i}(x_i,\varepsilon))\leq e^{-r N^*}.
\]
By the arbitrariness of the disjoint family $\{\overline{B}_{n_i}(x_i,\varepsilon)\}_i$, it follows that
\[
\mathcal{M}^P(f_{1,\infty},K_{m,N},\varphi,\varepsilon,s+2r,N^*)=\sup\left\{\sum_i e^{-n_i(s+2r)+S_{n_i}\varphi(x_i)}\right \}\leq e^{- rN^*}.
\]
Letting $N^*\to \infty$, we obtain $\mathcal{M}^P(f_{1,\infty},K_{m,N},\varphi,\varepsilon,s+2r)=0$.
Consequently, 
\[
  M^{\widetilde{P}}(f_{1,\infty},K_m,\varphi,\varepsilon,s+2r)=0\quad \text{for all }\varepsilon\in (0,1/m),
\]
and hence $P^P(f_{1,\infty},K_m,\varphi)\leq s+2r$.
Since $K=\bigcup_{m \geq 1}K_m$, we conclude that
\[
  P^P(f_{1,\infty},K,\varphi)=\sup_m P^P(f_{1,\infty},K_m,\varphi)\leq s+2r.
\]
Finally, by the arbitrariness of $r$, it follows that $P^P(f_{1,\infty},K,\varphi)\leq s$.

(2) Similarly, fix $r>0$. In this part, $K$ is assumed to be a Borel set with $\mu(K)>0$, and for each $m \geq 1$ define
\[
K_m=\left\{x\in K: \limsup_{n\to\infty} \frac{1}{n}\left(-\log\mu(B_n(x,\varepsilon))+S_n\varphi(x)\right)>s-r\,\text{for all}\, \varepsilon\in (0,1/m)\right\}.
\]
Clearly, $K_m\subseteq K_{m+1}$, $K=\bigcup_{m=1}^{\infty} K_m$, and $\mu(K)=\lim_{m\to\infty} \mu(K_m)>0$.

Hence there exists $m\geq 1$ sufficiently large such that $\mu(K_m)>0$.
For such $m$, define for each $N \geq 1$
\[
K_{m,N}=\left\{x\in K_m:\limsup_{n\to\infty}\frac{1}{n}(-\log\mu(B_n(x,\varepsilon))+S_n\varphi(x))>s-r,\forall n\geq N \text{ and } \varepsilon\in (0,1/m)\right\}.
\]
It is clear that $K_{m,N}\subseteq K_{m,N+1}$, $K_m=\bigcup_{N=1}^{\infty}K_{m,N}$, and $\mu(K_m)=\lim_{N\to\infty} \mu(K_{m,N})$.

Since $\mu(K)=\lim_{m\to\infty}\lim_{N\to\infty}\mu(K_{m,N})>0$, we may choose $m$ and $N^*$ such that $\mu(K_{m,N^*})>0$.
Set $K^*=K_{m,N^*}$ and $\varepsilon^*=1/m$. By the choice of $K^*$, we have
\[
\mu(B_n(x,\varepsilon))\leq e^{-n(s-r)+S_n\varphi(x)},\quad \forall x\in K^*, 0<\varepsilon<\varepsilon^*, n\geq N.
\]
Let $N>N^*$ be sufficiently large, and let $\{\overline{B}_{n_i}(x_i,\varepsilon/5): x_i \in K^*\}_{i\geq 1}$ be an open cover of $K^*$ with $n_i\geq N\geq N^*$ and $\varepsilon\in (0,\varepsilon^*)$.
By the Vitali covering lemma (see \Cref{vitali}), there exists a finite or countable pairwise disjoint subfamily $\{\overline{B}_{n_i}(x_i,\varepsilon/5\}$ such that
\[
K^*\subseteq \bigcup_{i\geq 1}\overline{B}_{n_i}(x_i,\varepsilon/5) \subseteq \bigcup_{i\geq 1}\overline{B}_{n_i}(x_i,\varepsilon).
\]
Therefore,
\begin{align*}
\mathcal{M}^P(f_{1,\infty},K,\varphi,\varepsilon/5,s-r,N)&\geq \mathcal{M}^P(f_{1,\infty},K^*,\varphi,\varepsilon/5,s-r,N) \\
&\geq \sum_i e^{-n_i(s-r)+S_{n_i}\varphi(x_i)}\\
&\geq \sum_{i\geq 1}\mu(B_{n_i}(x_i,\varepsilon))\\
&\geq \mu(K^*)>0.
\end{align*}
Letting $N\to\infty$, we obtain 
\[
  \mathcal{M}^{\widetilde{P}}(f_{1,\infty},K,\varphi,\varepsilon/5,s-r)\geq \mu(K^*)> 0,
\] 
which implies 
\[
  P^P(f_{1,\infty},K,\varphi,\varepsilon/5)\geq s-r.
\] 
Letting $\varepsilon\to 0$, we conclude that $P^P(f_{1,\infty},K,\varphi)\geq s-r$.
Finally, by the arbitrariness of $r$, it follows that $P^P(f_{1,\infty},K,\varphi)\geq s$, which completes the proof.
\end{proof}

\begin{remark}
We also have the Billingsley type theorem for packing topological pressure when $\overline{P}_{\mu}(f_{1,\infty},x,\varphi)$ is replaced by $\underline{P}_{\mu}(f_{1,\infty},x,\varphi)$ in \Cref{Billingsley}.
\end{remark}

\begin{proposition}\label{pmupp}
Let $(X,f_{1,\infty})$ be an NDS, $\varphi\in C(X,\mathbb{R})$, and $\mu \in \mathcal{M}(X)$. Then
\[
\overline{P}_{\mu}(f_{1,\infty},\varphi)\leq P_{\mu}^P(f_{1,\infty},\varphi).
\]
\end{proposition}

\begin{proof}
For any $s< \overline{P}_{\mu}(f_{1,\infty},\varphi)=\int_X \overline{P}_{\mu}(f_{1,\infty},x,\varphi)\, d\mu(x)$, there exist $\varepsilon>0$, $\epsilon>0$, and a Borel set $A \subseteq X$ with $\mu(A)>0$ such that
\begin{equation}\label{splusepsilon}
  \limsup_{n\to\infty}\frac{-\log \mu(B_n(x,\varepsilon))+S_n\varphi(x)}{n}>s+\epsilon\ \text{for all}\ x\in A.
\end{equation}
Let $\delta\in (0,\mu(A))$, and let $Z \subseteq X$ be any set satisfying $\mu(Z)\geq 1-\delta$. 
Choose a countable collection $\{Z_i\}$ covering $Z$. 
Then 
\[
  \mu(\cup_i Z_i)\geq \mu(Z) \geq 1-\delta,
\] 
and hence 
\[
  \mu(A\cap (\cup_i Z_i))\geq \mu(A)-\delta>0.
\]
Therefore, there exists an index $i$ such that $\mu(A\cap Z_i)>0$. 
It is clear that \eqref{splusepsilon} also holds for all $x\in A\cap Z_i$.
Applying the second assertion of \Cref{Billingsley}, we obtain 
\[
  P^P(f_{1,\infty},A\cap Z_i,\varphi) \geq s.
\]
Equivalently, 
\[
  \mathcal{M}^{\widetilde{P}}(f_{1,\infty},A\cap Z_i,\varphi,\varepsilon/5,s)=\infty.
\] 
It follows that
\[
  \mathcal{M}^{P}(f_{1,\infty},Z_i,\varphi,\varepsilon/5,s)\geq \mathcal{M}^{P}(f_{1,\infty},A\cap Z_i,\varphi,\varepsilon/5,s)>0.
\]
Consequently, $\mathcal{M}^{\widetilde{P}}(f_{1,\infty},Z,\varphi,\varepsilon/5,s)=\infty$,
which implies that $P^P(f_{1,\infty},Z,\varphi,\varepsilon/5)>s$.
Since $\mu(Z)\geq 1-\delta$, we conclude that $P^P(f_{1,\infty},\varphi)\geq s$.
By \Cref{pressureofmu}, this yields $P_{\mu}^P(f_{1,\infty},\varphi) \geq s$.
As $s< \overline{P}_{\mu}(f_{1,\infty},\varphi)$ is arbitrary, we finally obtain $P_{\mu}^P(f_{1,\infty},\varphi)\geq \overline{P}_{\mu}(f_{1,\infty},\varphi)$.
\end{proof}

\begin{remark}
By the Billingsley type theorem for Pesin topological pressure \cite[Theorem A]{Naz2024}, if there exists a Borel probability measure $\mu$ such that $\mu(K)>0$, and for all $x\in K$, $\underline{P}_{\mu}(f_{1,\infty},x,\varphi) \geq s$, it follows that $P^B(f_{1,\infty},K,\varphi)\geq s$. By applying a similar argument, we can also establish the analogous inequality $\underline{P}_{\mu}(f_{1,\infty},\varphi)\leq P_{\mu}^B(f_{1,\infty},\varphi)$.
\end{remark}

\begin{lemma}\label{lemma1}
Let $K \subseteq X$, $\varepsilon>0$, and $s>\|\varphi\|_{\infty}$. 
If $\mathcal{M}^P(f_{1,\infty},K,\varphi,\varepsilon,s)=\infty$, then for any $N\in \mathbb{N}$ and any finite interval $(a,b)$ with $a \geq 0$, there exists a finite pairwise disjoint collection of closed balls $\{\overline{B}_{n_i}(x_i,\varepsilon)\}_i$ such that $x_i\in K$, $n_i \geq N$, and $\sum_{i} e^{-sn_i+S_{n_i}\varphi(x_i)}\in (a,b)$.
\end{lemma}

\begin{proof}
Choose $N_1>N$ sufficiently large so that $e^{-N_1(s-\|\varphi\|_{\infty})}<b-a$.
Since $\mathcal{M}^P(f_{1,\infty},K,\varphi,\varepsilon,s)=\infty$, it follows from the definition that $\mathcal{M}^P(f_{1,\infty},K,\varphi,\varepsilon,s,N_1)=+\infty$.
By the definition of the packing topological pressure, there exists a finite pairwise disjoint collection of closed balls $\{\overline{B}_{n_i}(x_i,\varepsilon)\}_i$ such that $x_i\in K$, $n_i \geq N_1$ for all $i$, and $\sum_{i}e^{-sn_i+S_{n_i}\varphi(x_i)}>b$. 
Noting that $S_{n_i}\varphi(x_i)\leq n_i \|\varphi\|_{\infty}$, we obtain
\[
e^{-sn_i+S_{n_i}\varphi(x_i)}\leq e^{-n_i(s-\|\varphi\|_{\infty})}\leq e^{-N_1(s-\|\varphi\|_{\infty})} <b-a.
\]
Therefore, by discarding balls from the collection $\{\overline{B}_{n_i}(x_i,\varepsilon)\}_i$ one at a time if necessary, we may arrange that $a<\sum_ie^{-sn_i+S_{n_i}\varphi(x_i)}<b$.
\end{proof}

A \emph{Polish space} is a separable topological space that is equipped with a complete metric.
Given a Polish space $X$, a subset $Z \subseteq X$ is said to be \emph{analytic} if there exists a Polish space $Y$ and a continuous map $f: Y\to X$ such that $Z=f(Y)$.
It is worth noting that every compact metric space is Polish, and that any nonempty Borel set is a continuous image of the set of natural numbers $\mathbb{N}$. Consequently, every Borel set is analytic. With these preliminaries in place, we are ready to prove the reverse inequalities in \Cref{main} using the tools from analytic set theory.

\begin{proposition}\label{upperbound}
Let $(X,f_{1,\infty})$ be an NDS and $\varphi\in C(X, \mathbb{R})$. 
If $Z$ is an analytic subset of $X$ with $P^P(f_{1,\infty},Z,\varphi)\geq \|\varphi\|_{\infty}:=\sup_{x\in X}|\varphi(x)|$, then for any $s\in (\|\varphi\|_{\infty},P^P(f_{1,\infty},Z,\varphi))$, there exist a compact subset $K \subseteq Z$ and a measure $\mu\in \mathcal{M}(K)$ such that $\overline{P}_{\mu}(f_{1,\infty},K,\varphi) \geq s$ and $P_{\mu}^P(f_{1,\infty},\varphi)\geq s$.
\end{proposition}

\begin{proof}
We adapt the proof of \cite{FH2012} and take advantage of ideas from \cite{ZZ2023,ZC2023}.
Our goal is to inductively construct a sequence $\{K_i\}_i$ of finite subsets of $Z$ and a sequence of finite measures $\{\mu_i\}_i$ such that $\mu_i$ is supported on $K_i$ for each $i$.
Additionally, we introduce a sequence of integers $\{n_i\}_i \subseteq \mathbb{N}$, a sequence of positive numbers $\{\gamma_i\}_i$, and a sequence of integer-valued functions $m_i: K_i\to \mathbb{N}$. 
The desired measure $\mu$ and subset $K \subseteq Z$ will be defined in terms of the sequences $\{\mu_i\}_i$ and $\{K_i\}_i$. 
To achieve this, we proceed in the following three steps:
\begin{enumerate}
    \item[(1)] Step 1: Construct the fundamental framework of $K_1$, $\mu_1$, $m_1(\cdot )$, $n_1$, and $\gamma_1$;
    \item[(2)] Step 2: Construct $K_2$, $\mu_2$, $m_2(\cdot)$, $n_2$, and $\gamma_2$;
    \item[(3)] Step 3: Assuming that $K_i$, $\mu_i$, $m_i(\cdot)$, $n_i$, and $\gamma_i$ have been constructed for $i=1,2,\dots, p$, construct $K_{p+1}$, $\mu_{p+1}$, $m_{p+1}(\cdot)$, $n_{p+1}$, and $\gamma_{p+1}$.
\end{enumerate}

\medskip

\textbf{Step 1}. Construct $K_1$, $\mu_1$, $m_1(\cdot)$, $n_1$, and $\gamma_1$.

For any $s\in (\|\varphi\|_{\infty},P^P(f_{1,\infty},Z,\varphi))$, choose $\varepsilon>0$ sufficiently small that $s<P^P(f_{1,\infty},Z,\varphi,\varepsilon)$, and then fix $t\in (s,P^P(f_{1,\infty},Z,\varphi,\varepsilon))$. 
Note that $t<P^P(f_{1,\infty},Z,\varphi,\varepsilon)$, so $\mathcal{M}^P(f_{1,\infty},Z,\varphi,\varepsilon,t)=+\infty$. 
Let
\[
H=\bigcup\{G \subseteq X: G \text{ is an open set and } \mathcal{M}^P(f_{1,\infty},G \cap Z, \varphi,\varepsilon, t)=0\}.
\]
By the separability of $X$, $H=\bigcup_{i} G_i$ and $Z \cap H = \bigcup_{i} (Z \cap G_i)$. 
Hence
\[
\mathcal{M}^P(f_{1,\infty},Z \cap H,\varphi,\varepsilon,t)\leq \sum_{i}\mathcal{M}^P(f_{1,\infty},Z\cap G_i,\varphi,\varepsilon,t)=0.
\]
Therefore, we have
\begin{equation}\label{KcapH}
\mathcal{M}^P(f_{1,\infty},Z \cap H,\varphi,\varepsilon,t)=0.
\end{equation}

Let $Z'=Z \backslash H=Z \cap (X \backslash H)$. 
We then have the following \Cref{claim}, which will be used frequently.

\begin{claim}\label{claim}
For any open set $G \subseteq X$, either $Z' \cap G= \emptyset$ or $\mathcal{M}^P(f_{1,\infty},Z'\cap G,\varphi,\varepsilon,t)>0$.
\end{claim}

\textit{Proof of \Cref{claim}}. 
For any open set $G \subseteq X$, suppose $\mathcal{M}^P(f_{1,\infty},Z'\cap G,\varphi,\varepsilon, t)=0$. 
Since $Z=Z' \cup (Z \cap H )$, we have $Z \cap G \subseteq (Z' \cap G) \cup (Z \cap H)$. 
By \Cref{propertiesp} and \eqref{KcapH}, we then obtain
\begin{align*}
    \mathcal{M}^P(f_{1,\infty},Z \cap G,\varphi,\varepsilon,t)
  &\leq \mathcal{M}^P(f_{1,\infty},Z'\cap G,\varphi,\varepsilon,t)+\mathcal{M}^P(f_{1,\infty},Z \cap H,\varphi,\varepsilon,t)=0.
\end{align*}
This implies that $G \subseteq H$ by the definition of $H$, and hence $Z' \cap G=\emptyset$. 
This completes the proof of the claim.

\bigskip

Applying \Cref{propertiesp} again to $Z=Z'\cup(Z\cap H)$, and using \eqref{KcapH}, we obtain
\begin{align*}
\mathcal{M}^P(f_{1,\infty},Z,\varphi,\varepsilon,t) & \leq \mathcal{M}^P(f_{1,\infty},Z',\varphi,\varepsilon,t)+\mathcal{M}^P(f_{1,\infty},Z \cap H,\varphi,\varepsilon,t)\\
&\leq \mathcal{M}^P(f_{1,\infty},Z',\varphi,\varepsilon,t).
\end{align*}
Thus, $\mathcal{M}^P(f_{1,\infty},Z',\varphi,\varepsilon,t)=+\infty$. 
Since $s<t$, it follows that $\mathcal{M}^P(f_{1,\infty},Z',\varphi,\varepsilon,s)=+\infty$.

Next, we apply \Cref{lemma1} to find a finite set $K_1 \subseteq Z'$ and an integer-valued function $m_1(\cdot)$ defined on $K_1$ such that $\{\overline{B}_{m_1(x)}(x,\varepsilon)\}_{x\in K_1}$ is pairwise disjoint, and $\sum_{x\in K_1}e^{-m_1(x)s+S_{m_1(x)}\varphi(x)}\in (1,2)$.

We then define
\[
  \mu_1=\sum_{x\in K_1}e^{-m_1(x)s+S_{m_1(x)}\varphi(x)}\delta_x,
\]
where $\delta_x$ is the Dirac measure at $x$. 
We choose $\gamma_1>0$ small enough that for any function $z: K_1\to X$ satisfying $d(z(x),x)<\gamma_1$ for all $x\in K_1$, and for distinct $x,y\in K_1$,
\begin{equation}\label{appendix1}
\left( \overline{B}(z(x),\gamma_1) \cup \overline{B}_{m_1(x)}(z(x),\varepsilon)\right) \cap \left(\overline{B}(z(y),\gamma_1)\cup \overline{B}_{m_1(y)}(z(y),\varepsilon) \right)=\emptyset.
\end{equation}
From this, we also see that the elements of $\{\overline{B}(x,\gamma_1)\}_{x\in K_1}$ are pairwise disjoint. 
By contradiction and the triangle inequality, we also have
\[
 \overline{B}_{m_1(x)}(z(x),\varepsilon) \cap \overline{B}(y,\gamma_1) = \emptyset,\  \text{for all } y\in K_1\backslash \{x\}.
\]
Moreover, for any $x\in K_1$, we have $B(x,\gamma_1/4) \cap K_1 \neq \emptyset$. 
Since $K_1 \subseteq Z' \subseteq Z$ and $x\in Z' \cap B(x,\gamma_1/4)$, by \Cref{claim},
\[
  \mathcal{M}^P(f_{1,\infty},Z'\cap B(x,\gamma_1/4),\varphi,\varepsilon,t)>0,
\]
and thus
\[
  \mathcal{M}^P(f_{1,\infty},Z \cap B(x,\gamma_1/4),\varphi,\varepsilon,t) \geq \mathcal{M}^P(f_{1,\infty},Z' \cap B(x,\gamma_1/4),\varphi,\varepsilon,t)>0.
\]
Therefore, we can choose $n_1 \in \mathbb{N}$ sufficiently large such that $Z \supseteq Z_{n_1} \supseteq K_1$ and
\begin{equation}\label{appendix2}
\mathcal{M}^P(f_{1,\infty},Z_{n_1} \cap B(x,\gamma_1/4),\varphi,\varepsilon,t)>0, \ \text{for all } x\in K_1.
\end{equation}

\medskip

\textbf{Step 2}. Construct $K_2$, $\mu_2$, $m_2(\cdot)$, $n_2$, and $\gamma_2$.

Based on \eqref{appendix2}, and following an approach similar to that in step 1. Fix $x\in K_1$, we construct a finite set
\[
  E_2(x) \subseteq Z_{n_1}\cap B(x,\gamma_1/4)
\]
and an integer-valued function
\[
  m_2: E_2(x)\to \mathbb{N}\cap [\max \{m_1(y):y\in K_1\},+\infty)
\]
such that the following conditions hold:
\begin{enumerate}
  \item[(2a)] $\mathcal{M}^P(f_{1,\infty},Z_{n_1}\cap G,\varphi,\varepsilon,t)>0$ for any open set $G$ with $G \cap E_2(x) \neq \emptyset$;
  \item[(2b)] The elements of $\{\overline{B}_{m_2(y)}(y,\varepsilon)\}_{y\in E_2(x)}$ are pairwise disjoint, and
  \[
  \mu_1(\{x\})=e^{-m_1(x)s+S_{m_1(x)}\varphi(x)}<\sum_{y\in E_2(x)}e^{-m_2(y)s+S_{m_2(y)}\varphi(y)}<(1+2^{-2})\mu_1(\{x\}).
  \]
\end{enumerate}
To verify these, for $x\in K_1$ and define $F=Z_{n_1} \cap B(x,\gamma_1/4)$. 
Let
\[
  H_x =\bigcup \{G \subseteq X: G \text{ is an open set and } \mathcal{M}^P(f_{1,\infty},F\cap G,\varphi,\varepsilon,t)=0 \},
\]
and
\[
  F'=F\backslash H_x=F \cap (X\backslash H_x).
\]

As in Step 1, we can verify that $\mathcal{M}^P(f_{1,\infty},F',\varphi,\varepsilon,t)=\mathcal{M}^P(f_{1,\infty},F,\varphi,\varepsilon,t)>0$, and $\mathcal{M}^P(f_{1,\infty},F'\cap G,\varphi,\varepsilon,t)>0$ for any open set $G$ with $F' \cap G \neq \emptyset$.

Since $s<t$, the first equality implies that $\mathcal{M}^P(f_{1,\infty},F',\varphi,\varepsilon,s)=+\infty$. 
Applying \Cref{lemma1} again, we can find a finite set $E_2(x) \subseteq F'$, and an integer-valued function $m_2: E_2(x)\to \mathbb{N}\cap [\max\{m_1(y): y\in K_1\}, +\infty)$ such that condition (2b) holds. 
Observe that if $G$ is an open set with $G\cap E_2(x)\neq \emptyset$, then by the choice of $E_2(x)\subseteq F'$, we have $G \cap F' \neq \emptyset$. 
Furthermore, $\emptyset \neq G \cap F' \subseteq G \cap F \subseteq G \cap Z_{n_1}$, and by \Cref{claim}, we have
\[
\mathcal{M}^P(f_{1,\infty},G\cap Z_{n_1},\varphi,\varepsilon,t)\geq \mathcal{M}^P(f_{1,\infty},G\cap F',\varphi,\varepsilon,t)>0.
\]
Thus condition (2a) holds as well.
\medskip

Since $E_2(x) \subseteq Z_{n_1} \cap B(x,\gamma_1/4)$ for any $x\in K_1$, and the elements of $\{\overline{B}(x,\gamma_1/4)\}_{x\in K_1}$ are pairwise disjoint, it follows that for any distinct $x,x'\in K_1$, we have $E_2(x)\cap E_2(x')=\emptyset$. 
Now define
\[
  K_2=\bigcup_{x\in K_1}E_2(x), \quad \mu_2=\sum_{y\in K_2}e^{-m_2(y)s+S_{m_2(y)}\varphi(y)} \delta_y.
\]
It is clear that for any $x\in K_1$, $E_2(x)=K_2 \cap B(x,\gamma_1)$, and the elements of $\{\overline{B}_{m_2(x)}(x,\varepsilon)\}_{x\in K_2}$ are pairwise disjoint, as stated in \eqref{appendix1} and condition (2b). 
Thus, we can choose $\gamma_2\in (0,\gamma_1/4)$ small enough such that for any function $z: K_2\to X$ satisfying $d(z(x),x)<\gamma_2$ for all $x\in K_2$. We also have, for any distinct $x,y\in K_2$, 
\begin{equation}\label{appendix3}
\left( \overline{B}(z(x),\gamma_2)\cup \overline{B}_{m_2(x)}(z(x),\varepsilon) \right)\cap \left( \overline{B}(z(y),\gamma_2)\cup \overline{B}_{m_2(y)}(z(y),\varepsilon) \right)=\emptyset.
\end{equation}
From this, we see that the elements of $\{\overline{B}(x,\gamma_2)\}_{x\in K_2}$ are pairwise disjoint, and
\[
  \overline{B}_{m_2(x)}(z(x),\varepsilon)\cap \overline{B}(y,\gamma_2)= \emptyset, \ \text{for all } y\in K_2\backslash \{x\}.
\]
Finally, as in Step 1, we can find a large $n_2\in \mathbb{N}$ such that $Z \supseteq Z_{n_1,n_2}\supseteq K_2$ and
\[
  \mathcal{M}^P(f_{1,\infty},Z_{n_1,n_2}\cap B(x,\gamma_2/4),\varphi,\varepsilon,t)>0, \ \text{for all } x\in K_2.
\]

\medskip

\textbf{Step 3}. Assuming that $K_{i}$, $\mu_{i}$, $m_{i}(\cdot)$, $n_{i}$, and $\gamma_{i}$ have been constructed for $i=1,2,\dots,p$, construct $K_{p+1}$, $\mu_{p+1}$, $m_{p+1}(\cdot)$, $n_{p+1}$, and $\gamma_{p+1}$.

For any function $z: K_p\to X$ satisfying $d(z(x),x)<\gamma_p$, and any distinct $x,y \in K_p$, we have
\begin{equation}\label{appendix4}
\left ( \overline{B}(z(x),\gamma_p)\cup \overline{B}_{m_p(x)}(z(x),\varepsilon) \right) \cap \left (\overline{B}(z(y),\gamma_p)\cup \overline{B}_{m_p(y)}(z(y),\varepsilon) \right) =\emptyset,
\end{equation}
and 
\[
  \mathcal{M}^P(f_{1,\infty},Z_{n_1,n_2,\dots, n_p}\cap B(x,\gamma_p/4),\varphi,\varepsilon,t)>0 \quad \text{and}\quad Z \supseteq Z_{n_1,n_2,\dots, n_p} \supseteq K_p.
\]

Note that, by \eqref{appendix4}, the elements of $\{\overline{B}(x,\gamma_p)\}_{x\in K_p}$ are pairwise disjoint. 
For any $x\in K_p$, similar to Step 2, we can construct a finite set $E_{p+1}(x)\subseteq Z_{n_1,n_2,\dots, n_p} \cap B(x,\gamma_p/4)$ and an integer-valued map
\[
  m_{p+1}: E_{p+1}(x)\to \mathbb{N}\cap [\max\{m_p(y):y\in K_p \},+\infty)
\]
such that the following conditions hold:
\begin{enumerate}
  \item[(3a)] $\mathcal{M}^P(f_{1,\infty},Z_{n_1,n_2,\dots, n_p}\cap G,\varphi,\varepsilon,t)>0$ for any open set $G$ with $G \cap E_{p+1}(x)\neq \emptyset$;
  \item[(3b)] The elements of $\{\overline{B}_{m_{p+1}(y)}(y,\varepsilon)\}_{y\in E_{p+1}(x)}$ are disjoint and
  \[
  \mu_p(\{x\})<\sum_{y\in E_{p+1}(x)}e^{-m_{p+1}(y)s+S_{m_{p+1}(y)}\varphi(y)}<(1+2^{-(p+1)})\mu_p(\{x\}).
  \]
\end{enumerate}

Observe that for any $x\in K_p$, $E_{p+1}(x) \subseteq Z_{n_1,n_2,\dots, n_p} \cap B(x,\gamma_p/4)$, and the elements of $\{\overline{B}_{m_{p+1}(x)}(x,\gamma_p/4)\}_{x\in K_p}$ are also pairwise disjoint.
Therefore, for distinct $x,x'\in K_p$, we have $E_{p+1}(x)\cap E_{p+1}(x')=\emptyset$. 
We then define
\[
  K_{p+1}=\bigcup_{x\in K_p}E_{p+1}(x),\quad \mu_{p+1}=\sum_{y\in K_{p+1}}e^{-m_{p+1}(y)s+S_{m_{p+1}(y)}\varphi(y)}\delta_y.
\]
It is clear that $E_{p+1}(x)=K_{p+1} \cap B(x,\gamma_{p})$ for $x\in K_p$, and the elements of $\{\overline{B}_{m_{p+1}(x)}(x,\varepsilon)\}_{x\in K_{p+1}}$ are pairwise disjoint by \eqref{appendix4} and condition (3b). 
Thus, we can choose $\gamma_{p+1}\in (0,\gamma_p/4)$ small enough that for any function $z: K_{p+1}\to X$ satisfying $d(z(x),x)<\gamma_{p+1}$ for all $x\in K_{p+1}$, and for distinct $x,y\in K_{p+1}$, we have
\[
  \left(\overline{B}(z(x),\gamma_{p+1}) \cup \overline{B}_{m_{p+1}(x)}(z(x),\varepsilon) \right) \cap \left(\overline{B}(z(y),\gamma_{p+1})\cup \overline{B}_{m_{p+1}(y)}(z(y),\varepsilon) \right)=\emptyset.
\]
From this, the elements of $\{\overline{B}(x,\gamma_{p+1})\}_{x\in K_{p+1}}$ are pairwise disjoint, and
\[
  \overline{B}_{m_{p+1}(x)}(z(x),\varepsilon)\cap \overline{B}(y,\gamma_{p+1})=\emptyset, \ \text{for all } y\in K_{p+1}\backslash \{x\}.
\]
Finally, as in Step 2, we can find sufficiently large $n_{p+1}\in \mathbb{N}$ such that $Z\supseteq Z_{n_1,n_2,\dots, n_{p+1}}\supseteq K_{p+1}$, and
\[
  \mathcal{M}^P(f_{1,\infty},Z_{n_1,n_2,\dots, n_{p+1}}\cap B(x,\gamma_{p+1}/4),\varphi,\varepsilon,t)>0, \ \text{for all } x\in K_{p+1}.
\]

\medskip

By the previous three steps, we have completed the construction of $\{K_i\}_i$, $\{\mu_i\}_i$, $\{m_i(\cdot)\}_i$, $\{n_i\}_i$, and $\{\gamma_i\}_i$. Here are some basic properties:
\begin{enumerate}
  \item[(P1)] for each $i \geq 1$, the elements of $\mathcal{F}_i=\{\overline{B}(x,\gamma_i): x\in K_i\}$ are pairwise disjoint, and each element of $\mathcal{F}_{i+1}$ is a subset of $\overline{B}(x,\gamma_i/2)$ for some $x\in K_i$;
  \item[(P2)] for each $i \geq 1$, $K_i \subseteq Z_{n_1,n_2,\dots, n_i}$, and $\mu_i=\sum_{y\in K_i}e^{-m_i(y)s+S_{m_i(y)}\varphi(y)}\cdot \delta_y$ and $\mu_1(K_1)\in (1,2)$;
  \item[(P3)] for each $x\in K_i$, $E_{i+1}(x)=K_{i+1}\cap \overline{B}(x,\gamma_i)$, and for any $z\in \overline{B}(x,\gamma_i)$, we have
  \begin{equation}\label{appendix5}
  \overline{B}_{m_i(x)}(z,\varepsilon)\cap \bigcup_{y\in K_i\backslash \{x\}} \overline{B}(y,\gamma_i)=\emptyset,
  \end{equation}
  and
\begin{align}\label{appendix6}
\mu_i(\overline{B}(x,\gamma_i))&=e^{-m_i(x)s+S_{m_i(x)}\varphi(x)} \nonumber\\
  &\leq \sum_{y\in E_{i+1}(x)} e^{-m_{i+1}(y)s+S_{m_{i+1}(y)}\varphi(y)}\nonumber\\
  &\leq (1+2^{-(i+1)})\cdot\mu_i(\overline{B}(x,\gamma_i)).
  \end{align}
\end{enumerate}

\medskip

Next, we will construct a compact subset $K \subseteq Z$ and $\mu\in \mathcal{M}(K)$.

For any $x\in K_i$ with $\overline{B}(x,\gamma_i)\in \mathcal{F}_i$, then
\[
  \mu_i(\overline{B}(x,\gamma_i))=e^{-m_i(x)s+S_{m_i(x)}\varphi(x)},
\]
and
\begin{align*}
\mu_{i+1}(\overline{B}(x,\gamma_i))&=\sum_{F\in \mathcal{F}_{i+1}:\, F \subseteq \overline{B}(x,\gamma_i)}\mu_{i+1}(F)\\
&=\sum_{\overline{B}(x',\gamma_{i+1}):\,\, x'\in K_{i+1},\overline{B}(x',\gamma_{i+1})\subseteq \overline{B}(x,\gamma_i)}\mu_{i+1}(\overline{B}(x',\gamma_{i+1})),
\end{align*}
which simplifies to
\begin{align*}
\mu_{i+1}(\overline{B}(x,\gamma_i))&\leq \sum_{y\in K_{i+1}\cap \overline{B}(x,\gamma_i)}e^{-m_{i+1}(y)s+S_{m_{i+1}(y)}\varphi(y)}\\
&=\sum_{y\in E_{i+1}(x)}e^{-m_{i+1}(y)s+S_{m_{i+1}(y)}\varphi(y)}.
\end{align*}
By \eqref{appendix6}, it follows that
\[
  \mu_i(\overline{B}(x,\gamma_i))\leq \mu_{i+1}(\overline{B}(x,\gamma_i))\leq (1+2^{-i-1})\cdot\mu_i(\overline{B}(x,\gamma_i)).
\]
By induction, we see that for any $j \geq i$ and $\overline{B}(x,\gamma_i) \in \mathcal{F}_i$, we have
\begin{equation}\label{appendix7}
\mu_i(\overline{B}(x,\gamma_i))\leq \mu_j(\overline{B}(x,\gamma_i))\leq \prod_{n=i+1}^j (1+2^{-n})\mu_i(\overline{B}(x,\gamma_i))\leq C\cdot \mu_i(\overline{B}(x,\gamma_i)),
\end{equation}
where $C=\prod_{n=1}^{\infty}(1+2^{-n})<+\infty$.

Let $\overline{\mu}$ be the limit point of $\mu_i$ in weak$^{*}$-topology, and
\[
  K=\bigcap_{n=1}^{\infty}\overline{\bigcup_{i\geq n} K_i}.
\]
Then, $\overline{\mu}$ is supported on $K$, $K \subseteq \bigcap_{p=1}^{\infty}\overline{Z_{n_1,n_2,\dots, n_p}}\subseteq Z$, and
\begin{equation}\label{Zsubseteq}
K \subseteq \bigcup_{x\in K_i}B(x,\gamma_i/2).
\end{equation}

Following Cantor's diagonal method (see \cite[Proposition 3.5]{ZZ2023}), we can check that
\[
  \bigcap_{p=1}^{\infty}Z_{n_1,n_2,\dots, n_p}=\bigcap_{p=1}^{\infty} \overline{Z_{n_1,n_2,\dots, n_p}}.
\]
Thus, $K \subseteq Z$, and $K$ is compact.

For any $x\in K_i$, by \eqref{appendix7}, we have
\begin{equation}\label{overlinemu}
e^{-m_i(x)(s-\|\varphi\|_{\infty})}\leq e^{-m_i(x)s+S_{m_i(x)}\varphi(x)}=\mu_i(\overline{B}(x,\gamma_i))\leq \overline{\mu}(\overline{B}(x,\gamma_i))\leq C\cdot \mu_i(\overline{B}(x,\gamma_i)).
\end{equation}
For $i = 1$, it follows that
\begin{equation}\label{muZ}
1\leq \sum_{x\in K_1}\mu_1(B(x,\gamma_1))\leq \sum_{x\in K_1}\overline{\mu}(B(x,\gamma_1))=\overline{\mu}(Z)\leq C \cdot \sum_{x\in K_1}\mu_1(B(x,\gamma_1))\leq 2C.
\end{equation}

Notice that for any $x\in K_i$ and $z\in \overline{B}(x,\gamma_i)$, by \eqref{appendix5} and \eqref{Zsubseteq}, it follows that $\overline{B}_{m_{i}(x)}(z,\varepsilon)\subseteq \overline{B}(x,\gamma_i/2)$. 
Therefore, by \eqref{Zsubseteq} again, for any $z\in K$ and $i\in \mathbb{N}$, there exists $x\in K_i$ such that $z\in \overline{B}(x,\gamma_i/2)$. 
Hence, by \eqref{overlinemu},
\begin{equation}\label{appendix8}
\overline{\mu}(\overline{B}_{m_i(x)}(z,\varepsilon))\leq \overline{\mu}(\overline{B}(x,\gamma_i/2))\leq \overline{\mu}(\overline{B}(x,\gamma_i))\leq C \cdot e^{-m_i(x)s+S_{m_i(x)}\varphi(x)}.
\end{equation}

Now we normalize $\mu$ as follows: $\mu(B):=\overline{\mu}(B)/\overline{\mu}(K)$ for any $B\in \mathcal{M}(X)$. 
Thus $\mu\in \mathcal{M}(K)$ and $\mu(K)=1$. Combining this with the fact that $P^P(f_{1,\infty},K,\varphi)>s$ and 
\[
  P_{\mu}^P(f_{1,\infty},\varphi)=\lim_{\delta\to0}\inf\{P^P(f_{1,\infty},K,\varphi):\mu(K)\geq 1-\delta\}.
\]
We conclude that $P_{\mu}^P(f_{1,\infty},\varphi) \geq s$. 
Hence, $P_{\mu}^P(f_{1,\infty},\varphi) \geq P^P(f_{1,\infty},K,\varphi)$.

On the other hand, for any $z\in K$, by \eqref{appendix8}, there exists an increasing sequence $\{k_i\}$ such that
\begin{equation}\label{normalized}
\mu(B_{k_i}(z,\varepsilon))\leq C \frac{e^{-k_i s+S_{k_i}\varphi(z)}}{\overline{\mu}(K)}.
\end{equation}

By the definition of upper measure-theoretic pressure (see \Cref{definitiontheoretic}), we have 
\begin{align*}
\overline{P}_{\mu}(f_{1,\infty},K,\varphi)&=\int_K \lim_{\varepsilon\to 0}\limsup_{n\to\infty}\frac{1}{n}\big[-\log \mu(B_n(x,\varepsilon))+S_n\varphi(x)\big]d \mu(x)\\
&\geq \int_K \limsup_{i\to\infty}\frac{1}{k_i}\left[ -\log \mu(B_{k_i}(x,\varepsilon))+S_{k_i}\varphi(x) \right] d \mu(x)\\
& \geq \int_K \limsup_{i\to\infty}\frac{1}{k_i} \left[ -\log \left(C\cdot \frac{e^{-k_i s+S_{k_i}\varphi(z)}}{\overline{\mu}(K)}\right)+S_{k_i}\varphi(x) \right]d \mu(x) \quad \text{(by \eqref{normalized})}\\
&=\int_K \limsup_{i\to\infty}\frac{1}{k_i}\big (-\log C + \log \overline{\mu}(K) + k_i s \big)d \mu(x)\quad (\text{by } \eqref{muZ}, \overline{\mu}(K) \text{ is finite})\\
&=s \cdot\mu(K)=s.
\end{align*}
Thus, we conclude that
\[
  \overline{P}_{\mu}(f_{1,\infty},K,\varphi)\geq P^P(f_{1,\infty},Z,\varphi)\geq P^P(f_{1,\infty},K,\varphi).
\]
Therefore, the proof is complete.
\end{proof}

\begin{theorem}[Variational principle for packing topological pressure]\label{main}
Let $(X,f_{1,\infty})$ be an NDS and $\varphi \in C(X,\mathbb{R})$. If $K$ is a nonempty subset of $X$ with $P^P(f_{1,\infty},K,\varphi)\geq \|\varphi\|_{\infty}$, then
    \begin{align*}
    P^P(f_{1,\infty},K,\varphi)
    &=\sup\{\overline{P}_{\mu}(f_{1,\infty},K,\varphi):\mu\in \mathcal{M}(X),\mu(K)=1 \}\\
    &=\sup\{P_{\mu}^P(f_{1,\infty},\varphi):\mu\in \mathcal{M}(X),\mu(K)=1 \}.
    \end{align*}
\end{theorem}

\begin{proof}
For any $\mu\in \mathcal{M}(X)$ with $\mu(K)=1$, combining \Cref{ppmeasure} and \Cref{pmupp}, we have
\[
P^P(f_{1,\infty},K,\varphi) \geq P_{\mu}^P(f_{1,\infty},\varphi) \geq \overline{P}_{\mu}(f_{1,\infty},\varphi).
\]
Observe that
\[
   \overline{P}_{\mu}(f_{1,\infty},\varphi)=\overline{P}_{\mu}(f_{1,\infty},X,\varphi)\geq \overline{P}_{\mu}(f_{1,\infty},K,\varphi).
\]
Thus, these two inequalities provide a lower bound for $P^P(f_{1,\infty},K,\varphi)$. On the other hand, by \Cref{upperbound}, we obtain the corresponding upper bound. Therefore, the proof is complete.
\end{proof}

\section{Applications}
\subsection{Measure-preserving systems}
Let $(X,\mathcal{B}(X),\mu)$ be a probability space, and let $f_{1,\infty}$ be a sequence of continuous maps from $X$ to itself.
Given $\mu\in \mathcal{M}(X)$, we say that $\mu$ is \emph{$f_{1,\infty}$-invariant} if $\mu(f_k^{-1}(A))=\mu(A)$ for each $f_k\in f_{1,\infty}$ and each Borel set $A \subseteq X$. That is, each map $f_k\in f_{1,\infty}$ preserves the same probability measure $\mu$.
We denote by $\mathcal{M}(X,f_{1,\infty})$ the set of $f_{1,\infty}$-invariant Borel probability measures on $X$. 
For a more detailed discussion of the motivation for studying the $f_{1,\infty}$-invariant measures, we refer the reader to \cite{ZLXZ2012}.

For each $f_k \in f_{1,\infty}$, define the push-forward operator$f_{k}^*:\mathcal{M}(X)\to \mathcal{M}(X)$ by $(f_{k}^*\mu)(B)=\mu(f_{k}^{-1}(B))$ for each Borel set $B \in \mathcal{B}(X)$. By an argument similar to that in \cite{Wal2000a}, one can verify that $f_{k}^*$ is continuous and affine. Moreover, $\mathcal{M}(X,f_{1,\infty})=\{\mu\in \mathcal{M}(X): f_{k}^*\mu=\mu \ \text{for all}\, f_k \in f_{1,\infty} \}$.

\begin{proposition}\label{mxff}
Let $(X,f_{1,\infty})$ be an NDS and let $\mu\in \mathcal{M}(X)$. Then
\begin{enumerate}
    \item[(1)] for every $f_k \in f_{1,\infty}$ and $\varphi\in C(X,\mathbb{R})$,
    \[
    \int \varphi d\big(f_{k}^*\mu\big)=\int \varphi\circ f_k d\mu.
    \]
    \item[(2)] $\mu\in \mathcal{M}(X,f_{1,\infty})$ if and only if, for every $f_k\in f_{1,\infty}$ and $\varphi\in C(X,\mathbb{R})$,
\begin{equation}\label{varphifk}
    \int \varphi\circ f_k d\mu=\int \varphi d\mu.
\end{equation}
 \item[(3)]  $\mathcal{M}(X,f_{1,\infty})$ is compact.
\end{enumerate}
\end{proposition}

\begin{proof}
The first two assertions follow directly from \cite{Wal2000a}. 
To prove the third assertion, let $\mu_n$ be a sequence in $\mathcal{M}(X,f_{1,\infty})$ such that $\mu_n \to \mu$. 
By \eqref{varphifk}, for every $\varphi\in C(X,\mathbb{R})$ and every $f_k\in f_{1,\infty}$, we have
\[
\int \varphi \circ f_k d \mu=\lim_{n\to\infty} \int \varphi \circ f_k d \mu_n= \lim_{n\to\infty} \int \varphi d \mu_n= \int \varphi d \mu.
\]
It then follows from assertion (2) that $\mu\in \mathcal{M}(X,f_{1,\infty})$. Therefore $\mathcal{M}(X,f_{1,\infty})$ is compact.
\end{proof}

\begin{proposition}\label{munconverge}
Let $(X,f_{1,\infty})$ be an NDS and $f_n$ converges uniformly to a limit $f$. If $\mu\in \mathcal{M}(X,f_{1,\infty})$, then $\mu\in \mathcal{M}(X,f)$.
\end{proposition}

\begin{proof}
For any $\varphi\in C(X,\mathbb{R})$, define
\[
w_{\varphi}(\delta)=\sup_{x,y\in X,\,d(x,y)<\delta}\left|\varphi(x)-\varphi(y)\right|.
\]
Given any $\varepsilon>0$, there exists $\delta>0$ such that $w_{\varphi}(\delta)<\varepsilon$.
Note that $f_n$ converges uniformly to $f$, hence $f$ is continuous. 
Moreover, for this $\delta>0$, there exists $N\geq 1$ such that for all $n\geq N$, $\sup_{x\in X}d(f_n(x), f(x))<\delta$.

Let $\mu\in \mathcal{M}(X,f_{1,\infty})$. Then for $n\geq N$,
\[
  \left|\int \varphi \circ f_n d \mu- \int \varphi\circ f d \mu \right|\leq \int \left | \varphi \circ f_n - \varphi\circ f \right |d \mu\leq w_{\varphi}(\delta) < \varepsilon.
\]
This indicates that
\[
  \lim_{n\to\infty} \int \varphi \circ f_n d \mu = \int \varphi\circ f d \mu.
\]
Combining this with \eqref{varphifk}, then
\[
   \int \varphi\circ f d \mu=\lim_{n\to\infty} \int \varphi \circ f_n d \mu =\lim_{n\to\infty} \int \varphi d \mu = \int \varphi d \mu.
\]
Thus $\mu\in \mathcal{M}(X,f)$.
\end{proof}

A point $x\in X$ is said to be \emph{non-wandering} for $f_{1,\infty}$ if for every neighborhood $U(x)$ of $x$, there exist positive integers $n$ and $k$ with $f_n^k(U(x))\cap U(x)\neq \emptyset$.
We denote the set of all non-wandering points of $f_{1,\infty}$ by $\Omega(f_{1,\infty})$.
A set $U\subset X$ is said to be \emph{wandering} if $f_n^k(U)\cap U=\emptyset$ for all integers $n$ and $k$.
A point $x\in X$ is wandering if it belongs to some wandering set \cite{KS1996}.

\begin{proposition}\label{nonwanering}
Let $(X,f_{1,\infty})$ be an NDS on the compact metric space $(X,d)$.
For any $\mu\in \mathcal{M}(X,f_{1,\infty})$, we have $\mu(\Omega(f_{1,\infty}))=1$.
\end{proposition}

\begin{proof}
Let $\{U_n\}_{n=1}^{\infty}$ be a base of the topology.
By \cite{Wal2000a}, $X\backslash\Omega(f_{1,\infty})$ is the union of those $U_n$ such that $U_n, f_1^{-1}(U_n),f_1^{-2}(U_n),\dots$ are pairwise disjoint.
Then $\mu(U_n)=0$ for every $\mu\in \mathcal{M}(X,f_{1,\infty})$. Thus $\mu(X\backslash\Omega(f_{1,\infty}))=0$ for all $\mu\in \mathcal{M}(X,f_{1,\infty})$.
\end{proof}

\begin{lemma} \label{naz2024}\cite[Theorem C]{Naz2024}
Let $(X,f_{1,\infty})$ be an NDS on a compact metric space $(X,d)$, and let $K$ be a nonempty compact subset of $X$.
Then, for any $\varphi\in C(X,\mathbb{R})$,
\[
P^B(f_{1,\infty},K,\varphi)=\sup\{\underline{P}_{\mu}(f_{1,\infty},\varphi): \mu\in \mathcal{M}(X), \mu(K)=1\}.
\]
\end{lemma}

\begin{corollary}
Let $(X,f_{1,\infty})$ be an NDS and let $\varphi\in C(X,\mathbb{R})$, then
\begin{align*}
P^B(f_{1,\infty},X,\varphi) &= P^B(f_{1,\infty},\Omega(f_{1,\infty}),\varphi),\\
P^P(f_{1,\infty},X,\varphi) &= P^P(f_{1,\infty},\Omega(f_{1,\infty}),\varphi).
\end{align*}
\end{corollary}

\begin{proof}
This result follows directly from \Cref{main}, \Cref{nonwanering}, and \Cref{naz2024}.
\end{proof}

\subsection{The set of generic points}
Let $(X,\mathcal{B}(X),\mu)$ be a probability space, $f_{1,\infty}$ a sequence of continuous selfmaps of $X$, and $\mu\in \mathcal{M}(X,f_{1,\infty})$.
A point $x\in X$ is said to be \emph{generic} for $\mu$ if the sequence of probability measures $\Gamma_n(x)=\frac{1}{n}\sum_{j=0}^{n-1}\delta_{f_1^j(x)}$ converges to $\mu$ in the weak$^*$-topology, where $\delta_y$ denotes the Dirac measure at $y$ \cite{Wan2021b}.
Denote by $G_{\mu}$ the set of all generic points for $\mu$.

Let $n\in \mathbb{N}$ and $F \subseteq \mathcal{M}(X,f_{1,\infty})$ be a neighborhood of $\mu$. Define
\[
X_{n,F}=\{x\in X: \Gamma_n (x)\in F\}, \quad G_{\mu}^m=\{x\in G_{\mu}: \Gamma_n(x) \in F, n\geq m\}.
\]
It is clear that $G_{\mu}=\bigcup_{m \geq 1}G_{\mu}^m$ and $G_{\mu}^m \subseteq X_{n,F}$ for $n \geq m$.

\begin{proposition}
Let $(X,f_{1,\infty})$ be an NDS and $\mu\in \mathcal{M}(X,f_{1,\infty})$.
Then for any $\varphi\in C(X,\mathbb{R})$,
\[
  P^P(f_{1,\infty},G_{\mu},\varphi) \leq \lim_{\varepsilon \to 0} \inf_{\mu\in F} \limsup_{n\to \infty} \frac{1}{n} \log P(f_{1,\infty},X_{n,F},\varphi,\varepsilon,n).
\]
\end{proposition}

\begin{proof}
We prove this result by \Cref{relationship}. 
For any $s<t< P^P(f_{1,\infty},G_{\mu}^m,\varphi)$, there exists $\gamma>0$ such that $P^P(f_{1,\infty},G_{\mu}^m,\varphi,\varepsilon)>t$ for all $\varepsilon\in (0,\gamma)$.
Consequently,
\[
  \mathcal{M}^{\widetilde{P}}(f_{1,\infty},G_{\mu}^m,\varphi,\varepsilon,t)=+\infty,
\]
which implies that $\mathcal{M}^P(f_{1,\infty},G_{\mu}^m,\varphi,\varepsilon,t)=+\infty$, and hence $\mathcal{M}^P(f_{1,\infty},G_{\mu}^m,\varphi,\varepsilon,t,N)=+\infty$ for all sufficiently large $N$.

Therefore, there exists a finite or countable collection of pairwise disjoint closed balls $\{\overline{B}_{n_i}(x_i,\varepsilon)\}$ with $x_i\in G_{\mu}^m$ such that
\[
  \sum_i e^{-sn_i+S_{n_i}\varphi(x_i)}>1,\quad \forall n_i \geq N.
\]
For each $j \geq N \geq m$, define $G_{\mu}^{m,j}=\{x_i\in G_{\mu}^m: n_i=j\}$. 
Then
\[
  \sum_{j \geq N}\sum_{x\in G_{\mu}^{m,j}}e^{-tj+S_k \varphi(x)}=\sum_{i \geq 1} e^{-sn_i+S_{n_i}\varphi(x_i)}>1.
\]
Hence, there exists $j \geq N$ such that
\[
  \sum_{x\in G_{\mu}^{m,j}}e^{-tj+S_k \varphi(x)} \geq e^{js}(1-e^{s-t}).
\]
Since the balls $\{\overline{B}_{n_i}(x_i,\varepsilon)\}$ are pairwise disjoint, the set $G_{\mu}^{m,j}$ is a $(j,\varepsilon)$-separated subset of $X_{j,F}$. 
It follows that
\[
  P(f_{1,\infty},X_{j,F},\varphi,\varepsilon,j) \geq e^{js}(1-e^{s-t}).
\]
Similar to \eqref{ejt}, we obtain
\[
  \limsup_{n\to \infty} \frac{1}{n}\log P(f_{1,\infty},X_{n,F},\varphi,\varepsilon,n) \geq s.
\]
By the arbitrariness of $s$ with $s<P^P(f_{1,\infty},G_{\mu}^m,\varphi)$, we conclude that
\[
  P^P(f_{1,\infty},G_{\mu}^m,\varphi)\leq \limsup_{n\to \infty} \frac{1}{n}\log P(f_{1,\infty},X_{n,F},\varphi,\varepsilon,n).
\]
Since $G_{\mu}=\bigcup_{m \geq 1} G_{\mu}^m$, it follows that 
\[
  P^P(f_{1,\infty},G_{\mu},\varphi)=\sup_m P^P(f_{1,\infty},G_{\mu}^m,\varphi)\leq \limsup_{n\to \infty} \frac{1}{n}\log P(f_{1,\infty},X_{n,F},\varphi,\varepsilon,n).
\]
Finally, letting $\varepsilon\to 0$, we have
\[
  P^P(f_{1,\infty},G_{\mu},\varphi) \leq \lim_{\varepsilon \to 0} \inf_{\mu\in F} \limsup_{n\to \infty} \frac{1}{n} \log P(f_{1,\infty},X_{n,F},\varphi,\varepsilon,n).
\]
\end{proof}

This proposition provides an upper bound for the packing topological pressure of the set of generic points.
However, this bound is far from satisfactory, as neither the Birkhoff ergodic theorem nor the Shannon-McMillan-Breiman theorem is available for NDS.
To the best of our knowledge, these theorems do not hold in general for NDS.
\section*{Acknowledgements}
The author would like to express great thanks to Prof. Li Jian and Dr. Carllos Eduardo Holanda for their kind comments and suggestions, and also thank the respected reviewers for their valuable comments.
This author is supported by the National Natural Science Foundation of China (No. 12271185) and the Basic and Applied Basic Research Foundation of Guangdong Province (Nos. 2022A1515011124 and 2023A1515110084).

\end{document}